\documentclass[11pt]{amsart}
\usepackage{amsmath,paralist,amssymb}
\usepackage{amsthm}
\usepackage{hyperref}
\usepackage{graphicx,subfigure}
\usepackage{tikz,color}
\usepackage[margin=1in]{geometry}
\usetikzlibrary{arrows,shapes,snakes,backgrounds} 
\usepackage{verbatim}

\usepackage[all]{xy}

\definecolor{DarkBlue}{rgb}{0,0,0.8} 
\definecolor{DarkGreen}{rgb}{0,0.5,0.0}

\newtheorem*{lemma*}{Lemma}

\newtheorem*{corollary*}{Corollary}
\newtheorem*{theorem*}{Theorem}
\newtheorem*{theorem1*}{Theorem \ref{thm:vol}}
\newtheorem*{theorem2*}{Theorem \ref{thm:vol2}}
\newtheorem*{theorem3*}{Theorem \ref{thm:volC}}

\numberwithin{equation}{section}
\newtheorem{theorem}{Theorem}[section]
\newtheorem{proposition}[theorem]{Proposition}
\newtheorem{lemma}[theorem]{Lemma}
\newtheorem{corollary}[theorem]{Corollary}
 
\theoremstyle{remark}
\newtheorem{remark}[theorem]{Remark}

\newtheorem{example}[theorem]{Example}

\theoremstyle{definition}
\newtheorem{definition}[theorem]{Definition}

\newcommand\multiset[2]%
{\mathchoice{\left(\kern-0.5em{\binom{#1}{#2}}\kern-0.5em\right)}
            {\bigl(\kern-0.3em{\binom{#1}{#2}}\kern-0.3em\bigr)}
            {\bigl(\kern-0.3em{\binom{#1}{#2}}\kern-0.3em\bigr)}
            {\bigl(\kern-0.3em{\binom{#1}{#2}}\kern-0.3em\bigr)}}

\DeclareMathOperator{\aff}{aff}
\DeclareMathOperator{\PS}{PS}

\def\sfe{{\sf e}}

   \def\vol{{\rm vol}}

\def\i{{\rm in }}

\def\m{{\bf m}}
 \def\f_H{{\bf w}}
 \def\f{{\bf f}}
 \def\a{{\bf a}}
\def\b{{\bf b}}
\def\R{\mathbb{R}}

\def\Z{\mathbb{Z}}
\def\g{{\bf a}}

 \def\F{\mathcal{F}}
\def\O{\mathcal{O}}
\def\I{\mathcal{I}}

   \def\vol{{\rm vol}}

\def\aa{{\bf a}}
 \def\f_H{{\bf w}}
 \def\f{{\bf f}}
 \def\a{{\bf a}}
\def\R{\mathbb{R}}

\def\Z{\mathbb{Z}}
\def\g{{\bf a}}

 \def\F{\mathcal{F}}
\def\O{\mathcal{O}}
\def\I{\mathcal{I}}

   \def\vol{{\rm vol}}

\def\ind{{\rm ind }}
\def\out{{\rm outd}}
\def\om{{\rm out}}

\def\aa{{\bf a}}
 \def\f_H{{\bf w}}
 \def\f{{\bf f}}
 \def\a{{\bf a}}

\def\R{\mathbb{R}}

\def\Z{\mathbb{Z}}
\def\g{{\bf a}}

 \def\F{\mathcal{F}}
\def\O{\mathcal{O}}
\def\I{\mathcal{I}}

\newcommand{\be}{\begin{equation}}

\newcommand{\eee}{\end{equation}}

\newcommand{\bd}{\begin{definition}}
\newcommand{\ed}{\end{definition}}
\newcommand{\bt}{\begin{theorem}}
\newcommand{\et}{\end{theorem}}
\newcommand{\bl}{\begin{lemma}}
\newcommand{\el}{\end{lemma}}
\newcommand{\bp}{\begin{proposition}}
\newcommand{\ep}{\end{proposition}}
\newcommand{\bc}{\begin{corollary}}
\newcommand{\ec}{\end{corollary}}

\def\R{\mathbb{R}}

\begin{document}

\tikzstyle{w}=[label=right:$\textcolor{red}{\cdots}$] 
\tikzstyle{b}=[label=right:$\cdot\,\textcolor{red}{\cdot}\,\cdot$] 
\tikzstyle{bb}=[circle,draw=black!90,fill=black!100,thick,inner sep=1pt,minimum width=3pt] 
\tikzstyle{bb2}=[circle,draw=black!90,fill=black!100,thick,inner sep=1pt,minimum width=2pt] 
\tikzstyle{b2}=[label=right:$\cdots$] 
\tikzstyle{w2}=[]
\tikzstyle{vw}=[label=above:$\textcolor{red}{\vdots}$] 
\tikzstyle{vb}=[label=above:$\vdots$] 

\tikzstyle{level 1}=[level distance=3.5cm, sibling distance=3.5cm]
\tikzstyle{level 2}=[level distance=3.5cm, sibling distance=2cm]

\tikzstyle{bag} = [text width=4em, text centered]
\tikzstyle{end} = [circle, minimum width=3pt,fill, inner sep=0pt]

\title[Volumes and Ehrhart polynomials of flow polytopes]{Volumes and Ehrhart polynomials of flow polytopes}
\author{Karola M\'esz\'aros}
\address{Karola M\'esz\'aros, Department of Mathematics, Cornell University, Ithaca NY 14853  \newline karola@math.cornell.edu
}
\author{Alejandro H.  Morales}
\address{Alejandro H. Morales,
Department of Mathematics and Statistics, University of Massachusetts, Amherst MA 01003 \newline ahmorales@math.umass.edu
}
\thanks{KM is partially supported by a National Science
  Foundation Grant  (DMS 1501059). AHM was partially supported by an
AMS-Simons travel grant.}
\date{\today}

\maketitle
\vspace{-.15in}
\begin{center}
  \emph{Dedicated to the memory of Bertram Kostant}
\end{center}

\begin{abstract} The Lidskii formula for the type $A_n$ root system
  expresses the volume and Ehrhart polynomial of the flow polytope of
  the complete graph with nonnegative integer netflows
  in terms of Kostant partition functions. For every integer polytope
  the volume is the leading coefficient of the Ehrhart polynomial.
  The beauty of the Lidskii formula is the revelation that for these
  polytopes their Ehrhart polynomial function can be deduced from their
  volume function! Baldoni and Vergne generalized Lidskii's result for
  flow polytopes of arbitrary graphs $G$ and nonnegative integer
  netflows. While their formulas are combinatorial in nature, their
  proofs are based on residue computations. In this paper we construct
  canonical polytopal subdivisions of flow polytopes which we use to prove the Baldoni--Vergne--Lidskii formulas. In contrast with the  original computational proof of these formulas, our proof  reveal their geometry and combinatorics. We conclude by exhibiting enumerative properties of the Lidskii formulas via our  canonical polytopal subdivisions. 
\end{abstract}

\section{Introduction}
\label{sec:intro}

Flow polytopes are a well studied \cite{BV2, BSDL,gallo} and rich family of polytopes that include the
Pitman--Stanley polytope \cite{PS}, the Chan--Robbins--Yuen polytope
\cite{CRY} and  the Tesler polytope \cite{MMR}; see  \cite{AIM, CKM,
  MSZ} for more examples. Flow polytopes have been shown to have close
connections with  representation theory \cite{BV2}, diagonal harmonics
\cite{MMR} and Schubert polynomials \cite{MStD}, among others. Two
fundamental questions about any integer polytope $\mathcal{P}$,
including flow polytopes,  are: What is the volume of $\mathcal{P}$?
What is the Ehrhart polynomial  of $\mathcal{P}$? 

This paper is concerned with the  answers to these  question for the case of flow polytopes  $\F_G({\bf a})$ (defined in
Section \ref{sec:red}).  These questions were answered  by Lidskii
\cite{lidskii} for  $\F_{k_{n+1}}({\bf a})$, where $k_{n+1}$ denotes
the complete graph with $n+1$ vertices, and by Baldoni and Vergne \cite{BV2} for  $\F_G({\bf a})$, for arbitrary graphs $G$.   The Baldoni--Vergne proof relies on residue computations, leaving the combinatorial nature of their formulas a mystery. In this paper we demystify their beautiful formulas  appearing in Theorem \ref{thmA} below, by proving them via polytopal subdivisions of $\F_G({\bf a})$. We then  use the aforementioned polytopal subdivisions  to establish enumerative properties of the Baldoni--Vergne--Lidskii   formulas. For the   notation used in  Theorem \ref{thmA} consult Section \ref{sec:red}.

\begin{theorem}[Baldoni--Vergne--Lidskii  formulas {\cite[Thm. 38]{BV2}}]
\label{thmA} 
Let $G$ be a connected graph on the vertex set  $[n+1]$, with $m$
edges directed $i\to j$ if $i<j$, with at least one outgoing edge at
vertex $i$ for $i=1,\ldots,n$, and let
$\aa=(a_1,\ldots,a_n,-\sum_{i=1}^n a_i)$, $a_i \in \mathbb{Z}_{\geq 0}$. Then
\begin{align} \label{eq:vol}
\vol \F_G({\bf a}) &= \sum_{{\bf j}}
\binom{m-n}{j_1,\ldots,j_n} a_1^{j_1}\cdots
a_n^{j_n}\cdot K_{G}\left(j_1-\om_1, \ldots, j_n - \om_n,0\right),\\
K_{G}({\bf a}) &= \sum_{{{\bf j}}}
\binom{a_1+\om_1}{j_1}\cdots
\binom{a_{n}+\om_n}{j_{n}} \cdot  K_{G}\left(j_1-\om_1, \ldots, j_n -
                 \om_n,0\right), \label{eq:kost}\\
&= \sum_{{{\bf j}}}
\multiset{a_1-\i_1}{j_1}\cdots
\multiset{a_{n}-\i_n}{j_{n}} \cdot  K_{G}\left(j_1-\om_1, \ldots, j_n - \om_n,0\right), \label{eq:kostmultiset}    
\end{align} 
for $\om_i=\out_i-1$ and $\i_i=\ind_i-1$ where $\out_i$ and $\ind_i$ denote
the outdegree and indegree of  vertex $i$ in $G$. Each sum is over weak
compositions ${\bf j}=(j_1,j_2,\ldots,j_n)$ of $
m-n$ that are $\geq (\om_1,\ldots,\om_n)$ in dominance order and $\multiset{n}{k}:=\binom{n+k-1}{k}$.
\end{theorem}

In \eqref{eq:kost} $K_G({\bf a})$ denotes  the {\em Kostant partition function} of the graph $G$, which equals the number of lattice points of $\F_{G}({\bf a})$, as explained in    Section \ref{sec:red}.  The Ehrhart function of an integer polytope  $\mathcal{P}$  counts the number of lattice
points of the dilated polytope $t \mathcal{P}$, and it  is a
polynomial in $t$.  The coefficient of the highest
degree term of the Ehrhart  polynomial gives the volume of the polytope. The magic of the  Baldoni--Vergne--Lidskii formulas is that  for flow polytopes $\F_{G}({\bf a})$, their Ehrhart polynomial $K_G(t {\bf a})$ can be deduced from their volume function!

 The dominance order characterization of  the compositions {\bf j} in
 Theorem \ref{thmA} is  due to  Postnikov and Stanley \cite{St-trans}. Postnikov and Stanley also observed that  a proof of \eqref{eq:kost} can be obtained via  the judicious use of the {Elliott--MacMahon
  algorithm}  \cite{St-trans}.  We use  subdivisions of flow
polytopes to prove Theorem \ref{thmA},  explaining the summands  in
the RHS of \eqref{eq:vol}
and  \eqref{eq:kost} geometrically: each composition {\bf j} encodes
a type of cell of the subdivision, the Kostant partition function
encodes the number of times that type of cell appears in the
subdivision, the rest of the summand corresponds to the volume or lattice point contribution of that type of cell
(see Figure~\ref{fig:exflowpolys}).  To complete our polytopal proof of  \eqref{eq:kost}, we also need to  invoke the   Elliott--MacMahon
  algorithm, similar to the work of Postnikov and Stanley.  
  
   Our  subdivisions  of flow polytopes $\F_{G}({\bf a})$ generalize the Postnikov--Stanley subdivision 
 of the flow polytope $\F_G(1,0,\ldots,0,-1)$
(e.g. see \cite[\S 6]{MM}). 
 We refer to our subdivisions as the {\em
  canonical subdivision} of $\F_G(\a)$. We call the full dimensional
polytopes in the canonical subdivisions   \textit{cells}. We say that
two cells are of the same type if they are encoded by the same
composition ${\bf j}$. In Section~\ref{sec:numpieces} (see Theorems~\ref{cor:numtypescells} and \ref{thm:numcells}) we derive the following formulas for the
number of types of cells  and the number of cells
of the canonical subdivision of $\F_G(\a)$. 

\begin{theorem} \label{thm:intro_num_cells} 
Let $G$ be a graph with vertex set  $[n+1]$ and $\aa=(a_1,\ldots,a_n,-\sum_{i=1}^n a_i)$, $a_i \in \mathbb{Z}_{> 0}$. The number $N$ of types of cells in the
canonical subdivision of $\F_{G}(\a)$ is given by the determinant 
\[
N=\det \left[ \binom{\om_{i+1}+\cdots+\om_{n} + 1}{i-j+1} \right]_{1\leq i,j\leq n-1},
\]
and the number $M$ of cells of the canonical subdivision of $\F_{G}(\a)$ equals
\[
M = \vol \F_{G^\star}(1,0,\ldots,0,-1).
\]
where $G^{\star}$ is obtained
from $G$ by adding a vertex $0$ adjacent to vertices $i=1,2,\ldots,n$ of $G$.
\end{theorem}

We note that while Theorem \ref{thmA} is stated for outdegrees, there are analogues of \eqref{eq:vol} and
\eqref{eq:kost} in terms of indegrees of $G$ obtained by reversing the
digraph $G$.  We state the volume formula here.

\begin{corollary} \label{cor:vol_indegree}
Let $G$ be a graph on the vertex set  $[n+1]$ with $m$ edges directed $i\to j$ if $i<j$, with at least one incoming edge at
vertex $i$ for $i=2,\ldots,n+1$, and 
 $\b=(\sum_{i=1}^{n} b_i, -b_1,\ldots,-b_{n-1},-b_n)$ with  $b_i \in
 \mathbb{Z}_{\geq 0}$,
 for $i=1,\ldots,n$. Then 
\begin{equation} \label{eq:vol_indegree}
\vol \F_G(\b) = \sum_{{\bf j}}
\binom{m-n}{j_1,\ldots,j_n} b_1^{j_1}\cdots b_n^{j_n} \cdot K_G(0,\i_2-j_1,\ldots,\i_{n+1}-j_n),
\end{equation}
where $\i_i=\ind_i-1$ and $\ind_i$ is the indegree of vertex $i$ in
$G$, and the sum is over weak compositions ${\bf
  j}=(j_1,j_2,\ldots,j_n)$ of $m-n$ that are $\leq (\i_2,\ldots,\i_{n+1})$ in dominance order. 
\end{corollary}

Two important relations between the volume of a flow polytope and the number of lattice points of a related flow polytope can be deduced from the volume formulas \eqref{eq:vol} and \eqref{eq:vol_indegree} when we specialize to $\a=(1,0,\ldots,0,-1)$: 

\begin{corollary}[\cite{BV2,PS}] \label{cor:vol10-1case}
For a graph $G$ on the vertex set $[n+1]$ we have that 
\begin{align}
\vol \F_G(1,0,\ldots,0,-1) &= K_G(m-n-\om_1,-\om_2,\ldots,-\om_n,0), \label{eq:vol10-1caseOut}\\
&= K_G(0,\i_2,\i_3,\ldots,\i_n,-m+n+\i_{n+1}) \label{eq:vol10-1caseIn},
\end{align}
where $\om_i=\out_i-1$, $\i_i=\ind_i-1$ and $\out_i$, $\ind_i$ denote
the outdegree and indegree of  vertex $i$ in $G$. 
\end{corollary}

Thus, this corollary states that the volume of
$ \F_G(1,0,\ldots,0)$ equals the number of integer points in either the
polytope $\F_G(m-n-\om_1,-\om_2,\ldots,-\om_n,0)$ or $\F_G(0,\i_2,\i_3,\ldots,\i_n,m-n-\i_{n+1})$.
 
We highlight two families of flow polytopes with known product formulas for their volumes. Such formulas are obtained by applying Theorem~\ref{thmA}. 
\medskip

\noindent {\bf I. Pitman-Stanley polytopes:} Denote by $\Pi_n$ the graph on the vertex set $[n+1]$ and edges 
\[
E(\Pi_n) := \{(i,i+1),(i,n+1) \mid
i=1,\ldots,n\}.
\]
Baldoni and Vergne \cite[\S
3.6]{BV2} showed that the
polytope $\F_{\Pi_n}({\bf a})$ is integrally equivalent to the Pitman--Stanley polytope
\cite{PS}. They showed the Lidskii formulas in this case correspond
exactly to the volume and Ehrhart polynomial formulas in \cite{PS}
both involving Catalan many terms (in the notation of Theorem~\ref{thm:intro_num_cells} we have  $N = C_n:=\frac{1}{n+1}\binom{2n}{n}$).  Moreover, 
\[
\vol \F_{\Pi_n}({\bf a}) = n! \sum_{\bf j} 
\frac{a_1^{j_1}}{j_1!}\cdots \frac{a_n^{j_n}}{j_n!},
\]
where the sum is over the $C_n$ many tuples $(j_1,\ldots,j_n)$
satisfying $j_1+\cdots +j_n =n$ and with partial sums $j_1\geq 1,
j_1+j_2\geq 2,\ldots$.

\medskip

\noindent {\bf II. The Baldoni-Vergne polytopes:} When $G$ is the
complete graph $k_{n+1}$ with $n+1$ vertices the polytope
$\F_{k_{n+1}}(\a)$ was studied by Baldoni--Vergne \cite{BV2}. For
special values of $\a$ these polytopes have interesting volumes:

\begin{itemize}
\item[(a)] when ${\bf a} =
(1,0,\ldots,0,-1)$, the polytope $\F_{k_{n+1}}(\a)$  is called
the Chan-Robbins-Yuen (CRY) polytope \cite{CRY}. By \eqref{eq:vol10-1caseIn} we obtain
\[
\vol \F_{k_{n+1}}(1,0,\ldots,0,-1) =
K_{k_{n+1}}(0,0,1,2,\ldots,n-2,-{\textstyle \binom{n-1}{2}}).
\]
Zeilberger \cite{Z1}  showed that $K_{k_{n+1}}(0,0,1,2,\ldots,n-2,-{\textstyle \binom{n-1}{2}})$ is the
product of the first $n-1$ Catalan numbers as conjectured by Chan,
Robbins and Yuen \cite{CRY}:
\begin{equation} \label{eq:volCRY}
\vol \F_{k_{n+1}}(1,0,\ldots,0,-1) = C_0C_1\cdots C_{n-2}.
\end{equation}

\item[(b)] when ${\bf a} = (1,1,\ldots,1,-n)$,  the polytope $\mathcal{F}_{k_{n+1}}(\a)$ is called the Tesler polytope
\cite{MMR} whose lattice points correspond to Tesler matrices, of
interest in diagonal harmonics \cite{Hag}. Applying \eqref{eq:vol} to
this polytope yields
\[
\vol \F_{k_{n+1}}(1,1,\ldots,1,-n)= \sum_{{\bf j} }
\binom{\binom{n}{2}}{j_1,j_2,\ldots,j_n} \cdot K_{k_{n+1}}(j_1-n+1,j_2-n+2,\ldots,j_n,0).
\]
By Corollary~\ref{thm:numcellssubdiv}, the canonical subdivision of this polytope has
$M = \prod_{i=0}^{n-1} C_i$ cells.  In \cite{MMR} Rhoades and the authors showed that the volume equals
\begin{equation} \label{eq:volTesler}
\vol \F_{k_{n+1}}(1,1,\ldots,1,-n) = f^{(n-1,n-2,\ldots,1)}\cdot C_0C_1\cdots C_{n-1}
\end{equation}
where $f^{(n-1,n-2,\ldots,1)}$ is the number of 
standard Young tableaux of shape $(n-1,n-2,\ldots,1)$.
\item[(c)] when ${\bf a} = (1,1,0,\ldots,0,-2)$, the polytope $\mathcal{F}_{k_{n+1}}(\a)$ was studied by Corteel, Kim and the first author  \cite{CKM}. Applying \eqref{eq:vol} to
this polytope only the terms with compositions ${\bf j} =
(j_1,\binom{n}{2} - j_1,0,\ldots,0)$ survive. They then show that the volume equals 
\[
\vol \F_{k_{n+1}}(1,1,0,\ldots,0,-2) = 2^{\binom{n}{2}-1} C_0C_1\cdots C_{n-2}.
\]
\end{itemize}

The common theme of the proofs of volumes for the polytopes described in (a), (b) and (c) above is the application of  the Lidskii volume formula, followed by   variations of the {\em Morris constant term identity} \cite[Thm. 4.13]{WM},\cite{ZLF}.

\begin{figure}
\includegraphics{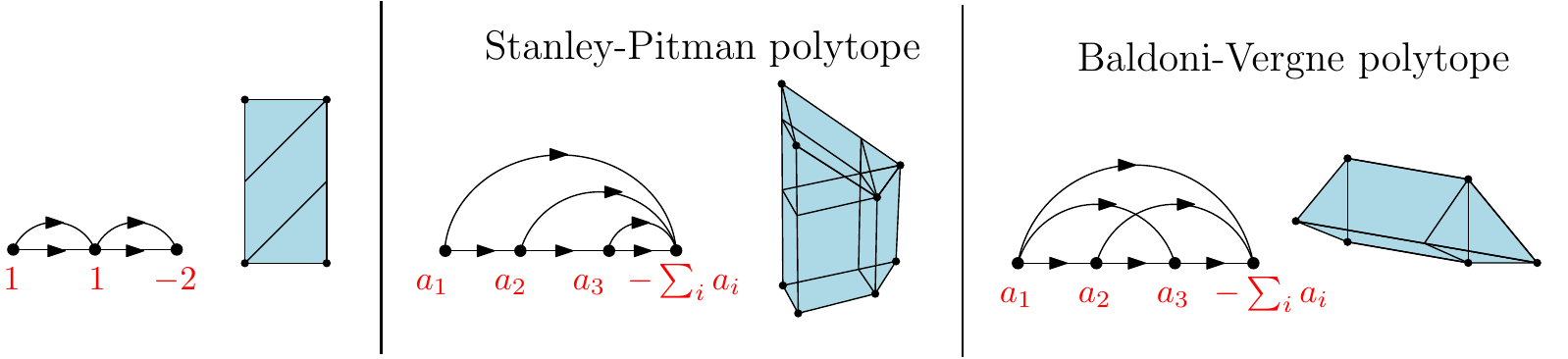}
\caption{Examples of graphs and their flow polytopes with the
  canonical subdivision. The second and thid example are instances of
  the Pitman--Stanley polytope and the Baldoni--Vergne polytope.}
\label{fig:exflowpolys}
\end{figure}

\subsection*{Outline} The outline of the paper is as follows. In Section \ref{sec:red} we explain the necessary definitions and background for flow polytopes. In Section \ref{sec:sub} we review the  subdivision of flow polytopes. In Sections \ref{sec:thm1} we prove \eqref{eq:vol}  via the canonical subdivision, while  in Section \ref{sec:kpf} we prove \eqref{eq:kost}. In Section~\ref{sec:numpieces} and \ref{sec:cayleytrick} we study the number of types of cells and the number of cells of subdivisions of flow polytopes with two different techniques: the canonical subdivision and the Cayley trick.

 \section{Flow polytopes $\F_G(\aa)$ and Kostant partition functions}
 \label{sec:red}

This section contains the background on flow polytopes and Kostant partition functions, following the exposition of \cite{MM}.  We also briefly revisit the Pitman--Stanley polytope mentioned in the introduction.

\label{subsec:flowpoly} Let $G$ be
a (loopless) directed acyclic connected graph on the vertex set $[n+1]$ with $m$ edges. To each edge $(i, j)$, $i< j$,  of $G$,  associate the positive
type $A_{n}$ root $\alpha(i,j)=e_i-e_j$.   Let $S_G := \{\{\alpha(e)\}\}_{e \in E(G)}$ be the
multiset of roots corresponding to the multiset of edges of $G$. Let  $M_G$ be the
$(n+1)\times m$ matrix whose columns are the vectors in $S_G$.  Fix an
integer  vector $\aa=(a_1,\ldots,a_n,-\sum_{i=1}^n a_i)$, $a_i \in
\mathbb{Z}_{\geq 0}$, referred to as the {\bf netflow}.  An {\bf
  $\a$-flow} $\f_G$ on $G$ is a vector $\f_G=(f(e))_{e \in E(G)}  \in
\R_{\geq 0}^{|E(G)|}$,  such that $M_G \f_G= \a$. That is, for all $1\leq i \leq n$, we have 

 \begin{equation} \label{eqn:flowA}
\sum_{e=(g,i) \in E(G)} f(e) + a_i = \sum_{e=(i,j) \in E(G)}
f(e)  \end{equation}
These equations imply that the netflow of vertex $n+1$ is
$-\sum_{i=1}^n a_i$.

Define the {\bf flow polytope} $\F_G(\g)$ associated to a  graph $G$ on the vertex set $[n+1]$ and the integer netflow vector $\aa$ as the set of all $\g$-flows $\f_G$ on $G$, i.e., $\F_G(\a)=\{\f_G \in \R^m_{\geq 0} \mid M_G \f_G = \g)\}$. If $\g$ is in the cone generated by
$S_G$ then $\F_G(\a)$ is not empty and if
$\g$ is in the interior of this cone then $\dim(\F_G(\a)) = m-n$
\cite[\S 1.1]{BV2}.

\medskip
The flow polytope $\F_G(\a)$ can be written as a Minkowski sum of flow
polytopes $\F_G(e_i-e_{n+1})$:

\begin{proposition}[{\cite[\S 3.4]{BV2}}] \label{prop:flow-minkowski}
For nonnegative integers $a_1,\ldots,a_n$ and $G$ a graph on the vertex set $[n+1]$ we have that
\begin{equation} \label{eq:flow-minkowski}
\F_G(\a) = a_1 \F_G(e_1-e_{n+1}) + a_2 \F_G(e_2-e_{n+1}) + \cdots +
a_n \F_G(e_n-e_{n+1}).
\end{equation}
\end{proposition}

\begin{proof}[Proof (sketch)]
By adding the flows edge-wise it follows that the Minkowski sum is
contained in $\F_G(\a)$. The other  inclusion can be shown 
by induction on the number of vertices with nonzero
netflow $a_i$. \end{proof}

The {\bf Kostant partition function}  $K_G$ evaluated at the vector $\a \in \Z^{n+1}$ is defined as

\begin{equation} \label{kost} K_G(\a)= \#  \Big\{ (f(e))_{e \in E(G)}
  \Bigm\vert  \sum_{e \in E(G)} f(e)  \alpha(e) =\a \textrm{ and } f(e) \in \Z_{\geq 0} \Big\},\end{equation}

\noindent where $\{\{\alpha(e)\}\}_{e\in E(G)}$ is the multiset of
positive roots corresponding to the multiset of
edges of $G$ defined above. In other words, $K_G(\a)$ is the number of ways to write the vector
$\a=(a_1,\ldots,a_n,-\sum_{i=1}^n a_i)$ as a $\mathbb{N}$-linear combination of the positive type $A_n$
roots $\alpha(e)$  corresponding to the edges of
$G$, without regard to order. Note that $K_G(\a)$ is the
number of lattice points of the flow polytope $\F_G(\a)$. 

The function $K_G(\a)$ is a piecewise polynomial function in
$a_1,a_2,\ldots,a_n$ (e.g. see \cite[Thm. 1.]{Sturmfels} and
\cite[Thm. 13]{BV2}). In fact, for vectors  
$(a_1,\ldots,a_n,-\sum_i a_i)$ in $\mathbb{Z}^{n+1}$ with  $a_i \geq
0$, the function $K_G(\a)$  is  a polynomial. 

\begin{proposition}[{\cite[Sec. 2.2]{BV2}}] \label{prop:polynomiality4kpf}
For $\a = (a_1,\ldots,a_n, -\sum_i a_i)$ in $\mathbb{Z}^{n+1}$ with
$a_i \geq 0$ for $i=1,\ldots,n$, the function $K_G(\a)$ is a
polynomial in $a_1,\ldots,a_n$.
\end{proposition}

The function $K_G(\a)$ has the following formal generating series:
\begin{equation}\label{gsKA}
\sum_{\g \in \Z^{n+1}} K_G(\a) x_1^{a_1}\cdots x_{n+1}^{-\sum_i a_i}
= {\prod_{(i,j)\in E(G)}} (1-x_i/x_j)^{-1},
\end{equation} 
where we order the variables $x_1<x_2<\ldots<x_{n+1}$ in order for the
expansion to be well defined.

By reversing the flow on a graph we obtain the following relation of
flow polytopes and the Kostant partition function. Given a directed graph $G$ with
vertices $[n+1]$ we denote
by $G^r$ the graph with vertices $[n+1]$ and edge $E(G^r)=\{(i,j) \mid
(n+2-j,n+2-i) \in E(G)\}$. That is, the graph obtained from $G$ by
reversing the edges and relabeling the vertices $i \mapsto n+1-i$. We say that two polytopes $P \subset \mathbb{R}^{n_1}$, $Q \subset \mathbb{R}^{n_2}$ are {\bf integrally equivalent} if there is an affine transformation $\varphi:\mathbb{R}^{n_1} \to \mathbb{R}^{n_2}$ that restricts to a bijection between  $P$ and $Q$ and between $\aff(P_1) \cap \mathbb{Z}^{n_1}$ and $\aff(Q) \cap \mathbb{Z}^{n_2}$. Integrally equivalent polytopes have the same face lattice, volume, and Ehrhart polynomials. We denote this equivalence by $P \equiv Q$.

\begin{proposition} \label{prop:symfp} For a graph $G$ on the vertex set $[n+1]$ and $(a_1,\ldots,a_n) \in \Z^n$: 
\[
\F_G(a_1,\ldots,a_n,-{\textstyle \sum_{i=1}^n a_i}) \equiv \F_{G^r}({\textstyle \sum_{i=1}^n a_i}, -a_{n},\ldots,-a_1).
\]
\end{proposition}

\begin{proof}
Given an $\a$-flow $\f_G=(f_e)_{e\in E(G)}$, let $\f_{G^r}=(f'(e))_{e\in E(G^r)}$ be the
flow defined by $f'(i,j) = f(n+2-j,n+2-i)$. Note that $\f_{G^r}$ is a
$\a^r$-flow where $\a^r=(\sum_{i=1}^n a_i, -a_n,\ldots,-a_1)$. The map $\f_G \mapsto
\f'_{G^r}$ is reversible and defines a correspondence between the
$\a$-flows and $\a^r$-flows. 
\end{proof}

If we restrict to counting integer points
in the two integrally equivalent polytopes in Proposition \ref{prop:symfp}, we obtain the following identity of Kostant partition functions:

\begin{corollary} \label{cor:symkpg}  For a graph $G$ on the vertex set $[n+1]$ and $(a_1,\ldots,a_n) \in \Z^n$: 

\[
K_G(a_1,\ldots,a_n,-{\textstyle \sum_{i=1}^n a_i}) = K_{G^r}({\textstyle \sum_{i=1}^n a_i}, -a_{n},\ldots,-a_1).
\]
\end{corollary}

\medskip

We end our background on flow polytopes by giving a characterization of the vertices of $\F_G(\a)$.

\begin{proposition}[{\cite[Lemma 2.1]{hille}}] \label{prop:charvert}
 The vertices of $\F_G(\a)$ are characterized as $\a$-flows
whose support yields a subgraph of $G$ with no (undirected) cycles.
\end{proposition}

As we will see, the flow polytope $\F_G(e_1-e_{n+1})$ is of particular interest. Their vertices are particularly easy to describe. Given a path ${\sf p}$ in $G$ from vertex $1$ to vertex $n+1$, let $\f({\sf p})$ be the unit flow with support in ${\sf p}$. 

\begin{corollary}[{\cite[Cor. 3.1]{gallo}}] \label{prop:vertFG1000}
The vertices of $\F_G(e_1-e_{n+1})$ are the unit flows $\f({\sf p})$ where ${\sf p}$ is a path in $G$ from vertex $1$ to vertex $n+1$.
\end{corollary}

We now sketch the proof that the Pitman--Stanley polytope (mentioned in the introduction) is a flow polytope. Recall that 
the Pitman--Stanley polytope is 
\[
\PS(a_1,\ldots,a_n) := \{ (x_1,\ldots,x_n) \in \mathbb{R}^n \mid x_i \geq 0, x_1 + \cdots + x_i \leq a_1 + \cdots + a_i \text{  for  } i =1,\ldots,n\},
\]
for parameters $a_1,\ldots,a_n$ with $a_i \geq 0$. This polytope was defined and studied in \cite{PS} and it is an important example of a {\em generalized permutahedron} \cite{P}. In \cite[Ex. 16]{BV1}, Baldoni and Vergne showed that this polytope is integrally equivalent to the flow polytope $\F_{\Pi_n}(\a)$ defined in the introduction:

\begin{proposition}[\cite{BV1}]
The polytopes $\F_{\Pi_n}(a_1,\ldots,a_n,-\sum_i a_i)$ and $\PS(a_1,\ldots,a_n)$ are integrally equivalent. 
\end{proposition}

\begin{proof}[Proof (sketch)]
The affine transformation $\varphi$ between the polytopes  $\PS(a_1,\ldots,a_n)$ and $\F_{\Pi_n}(\a)$ is defined as follows $\varphi:(x_1,\ldots,x_n) \mapsto \f_{\Pi_n}$ where 
\[
f(i,j) = \begin{cases}
x_i & \text{ if } j = n+1,\\
(a_1+\cdots + a_i) - (x_1+\cdots + x_i) & \text{ if } j=i+1.
\end{cases}
\]
\end{proof}

We note that when the parameters $a_i$ are positive integers the number of lattice points of $\PS(a_1,\ldots,a_n)$ counts certain {\em plane partitions} and is given by a determinant. 

\begin{theorem}[{\cite[Thm. 12]{PS}}] \label{thm:detlatptsPS}
For $(a_1,\ldots,a_n) \in \mathbb{N}^n$, the number of lattice points of the
Pitman--Stanley polytope $PS(a_1,\ldots,a_n)$ equals the number of plane partitions of shape $(a_1, a_1+a_2,\ldots, \sum_{i=1}^n a_i)$ with largest parts at most $2$. This number is given by the
determinant
\[
\#(\PS(a_1,\ldots,a_n) \cap \mathbb{Z}^n) = \det\left[
\binom{a_1+\cdots +a_{n-i+1}+1}{i-j+1}\right]_{1\leq i,j\leq n}.
\]
\end{theorem}

\section{Subdividing flow polytopes}
\label{sec:sub}
This section explains our method of subdividing flow polytopes. We explain basic and compounded reduction rules (Sections \ref{subsec:basicsubdiv} and \ref{subsec:subdiv} respectively), and characterize the polytopes obtained in a subdivision of $\F_G(\a)$ via these rules (Section \ref{sec:Gms}).

\subsection{Basic subdivision of flow
  polytopes} \label{subsec:basicsubdiv} 

 Given   a graph $G$ on the vertex set $[n+1]$ and   $(a, i), (i, b) \in E(G)$ for some $a<i<b$, let   $G_1$ and $G_2$ be graphs on the vertex set $[n+1]$ with edge sets
  \begin{align*} 
E(G_1)&=E(G)\backslash \{(i, b)\} \cup \{(a, b)\},  \\
E(G_2)&=E(G)\backslash \{(a, i)\} \cup \{(a, b)\}.
\end{align*}

\medskip

\noindent We refer to replacing $G$ by $G_1$ and $G_2$ as above as the \textbf{basic reduction}, or BR for short; see Figure \ref{fig:subdivrules}. The main result regarding the  basic reduction is as follows:

\begin{figure}
\begin{equation} \label{R1} \tag{BR}
\includegraphics[height=2.5cm]{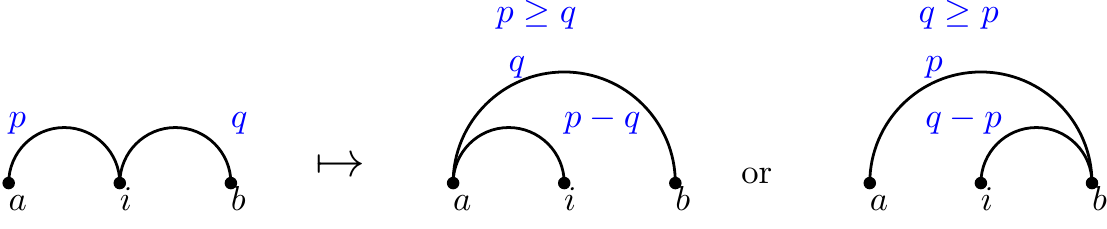}
\end{equation}
\caption{Basic reduction rule \eqref{R1}. The original edges have flow $p$ and $q$. The outcomes have reassigned flows to preserve the original netflow on the vertices.}
 \label{fig:subdivrules}
\end{figure}

 \begin{proposition}[Basic subdivision lemma] \label{red0} Given a  graph $G$  on the vertex set $[n+1]$,   $\aa \in \Z^{n}$, and two edges $e_1$ and $e_2$ of $G$ on which the basic reduction \eqref{R1}  can be performed yielding the graphs $G_1, G_2$, then 
 $$\F_G(\aa)= \mathcal{P}_1\bigcup \mathcal{P}_2 \quad   \text{ and } \quad \mathcal{P}_1^\circ \bigcap \mathcal{P}_2^\circ=\varnothing,$$
 
 \noindent where $ \mathcal{P}_i$ is integrally equivalent to  $\F_{G_i}(\aa)$, $ i \in [2]$, and  $\mathcal{P}^\circ$ denotes the interior of $\mathcal{P}$.  
\end{proposition}

The proof of Proposition \ref{red0} is left to the reader. See
\cite{MM, MStD} for proofs of this lemma. Remark \ref{rem:interpretationFGT} expands more on   the integral equivalence; by abuse of notation we will generally refer to $ \mathcal{P}_i$ in Proposition \ref{red0} as  $\F_{G_i}(\aa)$, for $i=1,2$.

  We can encode
a series of basic reductions on a flow polytope $\F_{G}(\a)$ in a rooted
tree called the \textbf{basic reduction tree}, or BRT for short; see Figure \ref{fig:brt} for an example. The root of this tree is the original graph
$G$.  After doing a BR on the edges $(a, i), (i, b)$, $a<i<b$, the descendant nodes of the root are the graphs
$G_1, G_2$ as above. For each new node we repeat this
process to define its descendants. If a node of this tree has a graph
$H$ with no edges $(a, i), (i, b)$, $a<i<b$,  then the node is a
{\bf leaf} of the BRT.  

\subsection{Compounded subdivision of flow polytopes} \label{subsec:subdiv}
Repeated use of the basic subdivision lemma (Proposition \ref{red0}) yields the canonical subdivision of flow polytopes as we explain in Section~\ref{sec:thm1}. In this section we state the compounded subdivision lemma (Proposition \ref{lem:big_red}), which is the result of applying the basic
reduction rules repeatedly on the incoming and outgoing edges of a
fixed vertex of $G$. The compounded subdivision lemma is a refinement of the subdivision lemma given  in
\cite[\S 5]{MM}. To state the result we introduce the necessary notation following \cite{MM}. 

A {\bf bipartite noncrossing tree}  is a  tree with a
distinguished bipartition of vertices into {\bf left vertices}
$x_1,\ldots,x_{\ell}$ and {\bf right vertices $x_{\ell+1},\ldots,
  x_{\ell+r}$} with no pair of edges $(x_p,x_{\ell+q}),
(x_t,x_{\ell+u})$ where $p<t$ and $q>u$. 
Denote by
$\mathcal{T}_{L,R}$ the set of  bipartite noncrossing trees
where $L$ and $R$ are the ordered
sets $(x_1,\ldots,x_{\ell})$ and $(x_{\ell+1},\ldots,x_{\ell+r})$ respectively. Note that $\#
\mathcal{T}_{L,R}=\binom{\ell+r-2}{\ell-1}$, since they are in
bijection with weak compositions of $r-1$ into $\ell$ parts. Namely, a
tree $T$ in $\mathcal{T}_{L,R}$ corresponds to the composition
$(b_1,\ldots,b_\ell)$ of $r -1$, where $b_i$ denotes
the number of edges incident to the left vertex $x_{\ell+i}$ in $T$ minus $1$. 

\begin{example}
The bipartite noncrossing tree encoded by the composition $(0,2,1,1)$ is the following:
\begin{center}
\includegraphics[width=1.3cm]{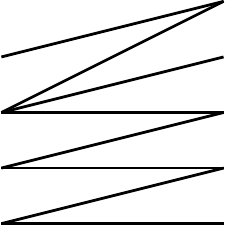}    
\end{center}
\end{example}

\begin{figure} 
\includegraphics[scale=.85]{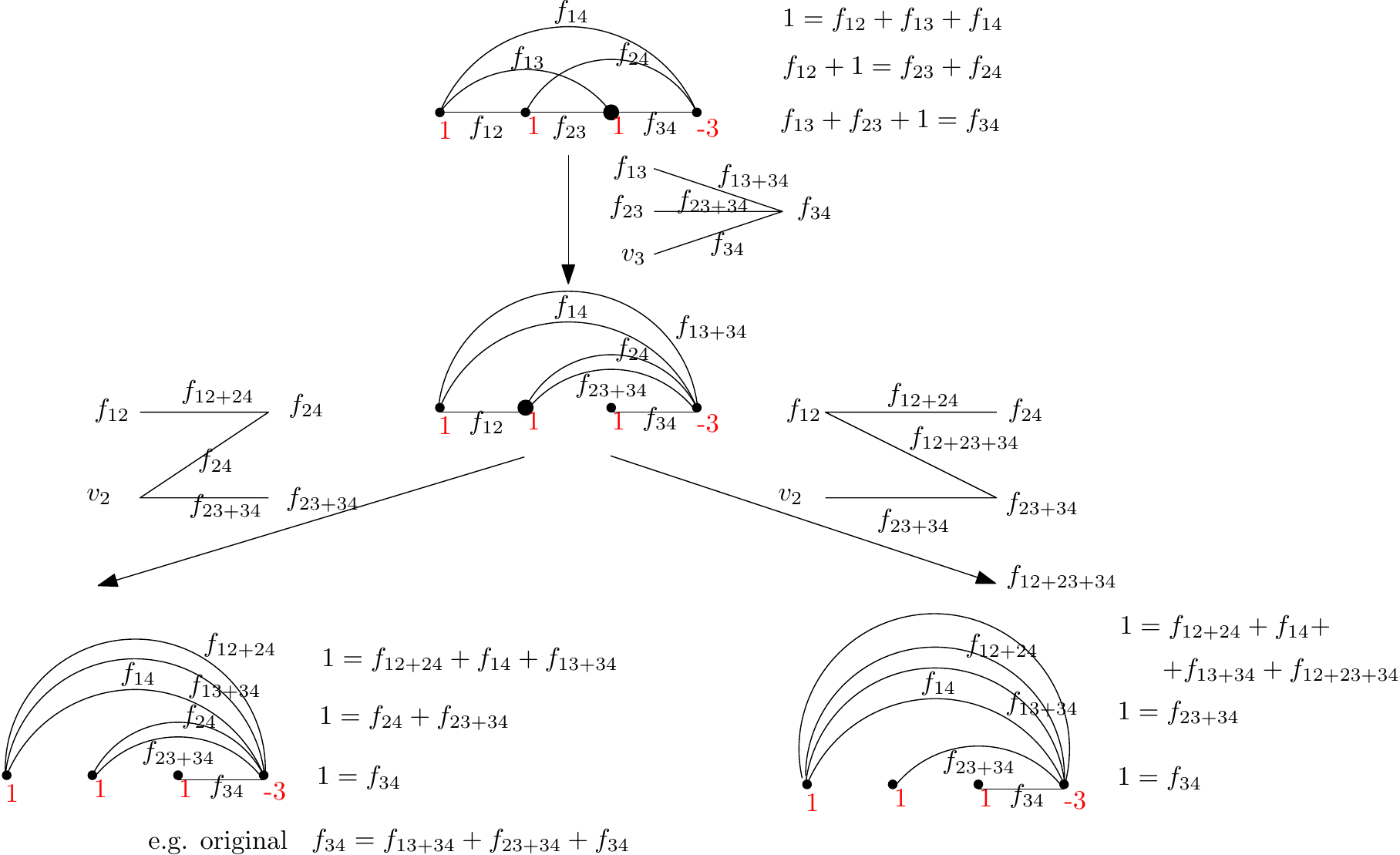}
\caption{Compounded reduction tree with change of variables indicated (see Remark \ref{rem:interpretationFGT}). The vertex
  of the graph
  where the compounded reduction is taking place is enlarged. The flow polytopes corresponding to the leaves of the compounded reduction tree  (CRT) subdivide the flow polytope corresponding to the root of the tree. Compare to the basic reduction tree of the same graph in Figure \ref{fig:brt}.}
\label{fig:complete}
\end{figure}

Consider a graph $G$ on the vertex set $[n+1]$ and an integer netflow
vector $\g=(a_1, \ldots, a_n, -\sum_i a_i)$.  Pick an arbitrary vertex
$i$,$1<i<n+1$, of $G$. There are two cases depending on whether $a_i =0$
or $a_i>0$.

\begin{itemize}
\item {\it Case 1: $a_i=0$.}  Given a  graph $G$ and one of its vertices $i$, let $\mathcal{I}_i=\mathcal{I}_i(G)$ be the multiset of {\bf incoming edges} to $i$, which are defined as edges of the form $(\cdot,i)$. Let $\mathcal{O}_i=\mathcal{O}_i(G)$ be the multiset of {\bf outgoing edges} from $i$, which are defined as  edges of the form   $(i,\cdot)$. Define $\ind_G(i):=\#\mathcal{I}_i(G)$ to be  the {\bf indegree} of vertex $i$ in $G$. 

Assign an ordering to the sets $\mathcal{I}_i$ and $\mathcal{O}_i$ and consider a tree $T \in \mathcal{T}_{\mathcal{I}_i,\mathcal{O}_i}$. For each tree-edge $(e_1,e_2)$ of $T$ where $e_1=(r,i) \in \mathcal{I}_i$ and $e_2=(i,s)\in \mathcal{O}_i$ let $edge(e_1,e_2)=(r,s)$. We think of $edge(e_1,e_2)$ as a formal {\bf sum of the edges} $e_1$ and $e_2$.

The graph $G^{(i)}_T$ is then defined as the graph obtained from $G$
by deleting all edges in $\I_i \cup \O_i$ of $G$  and adding  the
multiset of edges $\{\{edge(e_1,e_2) ~|~ (e_1,e_2)\in E(T)\}\}$, and
edge $(i,n+1)$.

\item  {\it Case 2: $a_i>0$.}  Instead of considering $T \in
\mathcal{T}_{\mathcal{I}_i,\mathcal{O}_i}$ we consider $T \in
\mathcal{T}_{\mathcal{I}_i \cup \{i\},\mathcal{O}_i}$. The edges of
$T$ are as in the previous case, with the exception that $edge(i,
(i,j))=(i,j)$. We define $G^{(i)}_T$  as the graph obtained from $G$
by deleting all edges in $\I_i \cup \O_i$ of $G$  and adding  the
multiset of edges of $T$. 
\end{itemize}

Note that in both cases, the graph $G^{(i)}_T$ has no incoming edges
to vertex $i$. See Figure \ref{fig:complete}.

\begin{remark} \label{rem:interpretationFGT}
We make the following precision when we refer to $\F_{G_T^{(i)}}(\a)$. Each edge of $G^{(i)}_T$ is a sum of (one or more) edges of  the
original graph $G$. As mentioned in Proposition~\ref{prop:charvert}, the vertices of
$\F_{G_T^{(i)}}(\a)$ are given by $\a$-flows on acyclic subgraphs of
${G_T^{(i)}}$. The acyclic subgraphs of ${G_T^{(i)}}$ can be mapped to
acyclic subgraphs of $G$ by mapping each edge $e$ of the
acyclic subgraph of ${G_T^{(i)}}$ to the edges in
$G$ that are formal summands of $e$. Moreover, with the previous map the $\a$-flows on acyclic
subgraphs of ${G_T^{(i)}}$ then map to $\a$-flows on acyclic subgraphs
of $G$. By abuse of notation when we refer to the flows in
$\F_{G_T^{(i)}}(\a)$ we interpret them  in the context of $G$. Thus we
define $\F_{G_T^{(i)}}(\a)$ as the convex hull 
of the $\a$-flows we obtain on $G$ as above. We do this so that
$\F_{G_T^{(i)}}(\a)\subseteq \F_{G}(\a)$. 
\end{remark}

\medskip

The proof of Theorem \ref{thmA} relies on the following lemma.

\begin{lemma}[Compounded subdivision lemma] \label{lem:big_red} Let  $G$ be a graph on the vertex set $[n+1]$.
 Fix an integer netflow vector $\aa=(a_1,\ldots,a_n,-\sum_{i=1}^n a_i)$, $a_i \in \mathbb{Z}_{\geq 0}$ and a vertex $i\in \{2,\ldots,n\}$ with incoming edges. Then, 
\begin{equation}
\F_G(\g)=\bigcup_{T \in \mathcal{T}_{L,R}} \F_{G_T^{(i)}}(\g),
\end{equation}
where 
\begin{equation}\label{tlr}
\mathcal{T}_{L,R}=\begin{cases}
\mathcal{T}_{\mathcal{I}_i,\mathcal{O}_i}  & \text{ if } a_i=0,\\
\mathcal{T}_{\mathcal{I}_i \cup \{i\},\mathcal{O}_i}& \text{ if }
a_i>0.
\end{cases}
\end{equation}
Moreover,  $\{\F_{G_T^{(i)}}(\g)\}_{T \in \mathcal{T}_{L,R}}$ are interior disjoint and of the same dimension as  $\F_{G}(\g)$. 

\end{lemma}

\begin{proof}
The case $a_i=0$ is proved in \cite[Lemma 5.4]{MM} where in our setup
$G_T^{(i)}$ has an edge $(i,n+1)$ with zero flow since $a_i=0$.  Next, we prove the case
$a_i>0$. 

Let $\widehat{G}$ be the graph obtained from $G$ by adding vertex $0$ and the
edge $(0,i)$ and
\[
\hat{\a}:=(a_i,a_1,\ldots,a_{i-1},0,a_{i+1},\ldots,a_n,-{\textstyle \sum_i
a_i}).
\]
The flow polytopes $\F_G(\a)$ and $\F_{\widehat{G}}(\hat{\a})$ integrally equivalent. This follows since any $\hat{\a}$-flow on $\F_{\widehat{G}}(\hat{\a})$ has flow $a_i$ on the  edge
$(0,i)$. Thus,  restricting any $\hat{\a}$-flow on $\F_{\widehat{G}}(\hat{\a})$  to the edges of
$G$ gives a flow in $\F_G(\a)$.   By applying
the subdivision lemma proved in \cite[Lemma 5.4]{MM} to $\F_{\widehat{G}}(\hat{a})$ on vertex $i$ with
zero flow we obtain
\[
\F_{\widehat{G}}(\hat{\g})=\bigcup_{T \in \mathcal{T}_{\hat{L},R}} \F_{\widehat{G}_T^{(i)}}(\hat{\g}).
\]
where $\hat{L}=\mathcal{I}_i(G) \cup \{v_0\}$, $R=\mathcal{O}_i(G)$ and $\{\F_{\widehat{G}_T^{(i)}}(\hat{\g})\}_{T \in \mathcal{T}_{L\cup
    \{v_0\}, R}}$ are interior disjoint and of the same dimension as
$\F_{\widehat{G}}(\hat{\a})$. Bipartite noncrossing trees $T$ in
$\mathcal{T}_{\mathcal{I}_i \cup \{v_0\}, \mathcal{O}_i}$ are in correspondence with trees $T'$
in $\mathcal{T}_{\mathcal{I}_i \cup \{v_i\},\mathcal{O}_i}$ by relabeling vertex $v_0$ to
$v_i$. Next, by identifying edges $edge((0,i),(i,j))$ (and their
flows) in $\widehat{G}^{(i)}_{T}$ (in
$\F_{\widehat{G}^{(i)}_T}(\hat{\a})$) with edges $edge(v_i,(i,j))$ (and
their flows) in $G^{(i)}_T$ (in $\F_{G_T^{(i)}}(\a)$) we see that the
  $\F_{\widehat{G}^{(i)}_T}(\hat{\a}) \equiv \F_{G_T^{(i)}}(\a)$ and 
\[
\F_G(\g)=\bigcup_{T \in \mathcal{T}_{\mathcal{I}_i \cup \{v_i\},\mathcal{O}_i}} \F_{G_T^{(i)}}(\g),
\]
 and
the polytopes $\F_{G^{(i)}_T}(\a)$ (interpreted as in Remark \ref{rem:interpretationFGT}) are interior disjoint and of the
same dimension as $\F_G(\a)$. 
\end{proof}

We refer to replacing $G$ by $\{G_T^{(i)}\}_{T \in \mathcal{T}_{L,R}}$
as in Lemma \ref{lem:big_red}  as a \textbf{compounded reduction}, or CR for short. We can encode
a series of compounded reductions on a flow polytope $\F_{G}(\a)$ in a rooted
tree called the \textbf{compounded reduction tree}, or CRT for short; see Figure \ref{fig:complete} for an example. The root of this tree is the original graph
$G$. After doing reductions on vertex $i$, the descendant nodes of the root are the graphs
$\F_{G^{(i)}_T}(\a)$ from the lemma. For each new node we repeat this
process to define its descendants. If a node of this tree has a graph
$H$ with no vertices $i=2,\ldots,n$ with  both incoming
and outgoing edges, then the node is a
{\bf leaf} of the reduction tree. Note that the flow polytopes
$\F_H(\a)$ of the graphs $H$ at the leaves of the tree have the same
dimension as $\F_G(\a)$.

\begin{example} 
Figure \ref{fig:complete} gives a CRT for the polytope
$\F_{k_4}(1,1,1,-3)$. The root of the reduction tree is labeled by the
complete graph $k_4$. Then we apply a compounded reduction at vertex $3$ to
obtain the graph $H:=([4], \{(1,2),(1,4),(1,4),(2,4),(2,4),
(3,4)\})$. On $H$ we do a CR at vertex $2$ yielding two
outcomes $H_1$ and $H_2$, drawn on the last row of the figure.  Note
that in both $H_1$ and $H_2$ there are no vertices with both incoming
and outgoing edges. This means we cannot do any more CR on
them. Such graphs are the leaves of this CRT. By Lemma \ref{lem:big_red}  the flow polytopes corresponding to the leaves of a CRT with root $G$ are a dissection of the flow polytope $\F_G(\a)$.  
\end{example}

\subsection{Subdividing  $\F_G(\a)$ into polytopes of known volume} \label{sec:Gms} 
The following lemma describes the   leaves of any compounded reduction tree
rooted at $G$. Given a tuple $\m=(m_1,\ldots,m_n)$ of positive
integers, let $G(\m)$ be the graph with vertices $[n+1]$ and $m_i$
edges $(i,n+1)$.

\bl \label{lem:leaves} Given the flow polytope $\F_G(\a)$ with $G$ a
graph on the vertex set $[n+1]$ and $a_i\geq 0$ for $i \in [n]$, the
leaves of any compounded reduction tree $R_G$ rooted at $G$ are graphs of the
form $G(\m)$ with $m_i=1$ if and only if $a_i=0$ and $\sum_{i=1}^n
m_i=\#E(G).$.
\el 

\begin{proof}
The result follows by iterating the compounded subdivision lemma (Lemma~\ref{lem:big_red}).   The leaves of $R_G$ will consist of graphs with no incoming
edges in vertices $i=2,\ldots,n$ such that their flow polytopes have same dimension as
$\F_G(\a)$. \end{proof}
 
\begin{remark}
We at times refer to the leaves described in Lemma \ref{lem:leaves} as
the \textbf{full dimensional leaves} of the CRT to emphasize that they yield flow
polytopes of the same dimension as the one we started with. This will
be in contrast with some of the leaves we obtain in Section~\ref{sec:kpf} in the basic reduction tree. 
\end{remark}

\begin{example}
The two leaves of the reduction tree in Figure~\ref{fig:complete} are
the graphs $G(3,2,1)$ and $G(4,1,1)$.
\end{example}

Next we calculate the volume of the polytopes $\F_{G(\m)}(\a)$.

\bl \label{lem:Gm} Given  $G(\m)$  on the vertex set $[n+1]$ with
$\m=(m_1, \ldots, m_n)$  a tuple of positive integers,  $\a=(a_1, \ldots, a_n) \in \Z_{\geq 0}^n$,  the normalized
volume of $\F_{G(\m)}(\a)$ is 
\begin{equation} \label{eq:volGm}
\vol(\F_{G(\m)}(\a))={{\#E(G(\m)) - n}\choose {m_1-1, \ldots,
    m_n-1}}a_1^{m_1-1}\cdots a_n^{m_n-1}. 
\end{equation}

\el

\proof 
The flow polytope $\F_{G(\m)}(\a)$ has dimension $\#E(G(\m))-n$ and is the product $\prod_{i=1}^n a_i\Delta_{m_i-1}$ of
dilated $(m_i-1)$-standard simplices $a_i\Delta_{m_i-1}$ each of
which has (standard) volume $a_i^{m_i-1}/(m_i-1)!$ 
\cite[Thm. 2.2]{BR}. Thus the
normalized volume of $\F_{G(\m)}(\a)$ is 
\begin{align*}
 \vol(\F_{G(\m)}(\a)) &= \dim(\F_{G(\m)}(\a))! \cdot \prod_{i=1}^n
 \frac{a_i^{m_i-1}}{(m_i-1)!} \\
&={{\#E(G(\m))-v}\choose {m_1-1, \ldots, m_n-1}}a_1^{m_1-1}\cdots a_v^{m_n-1}.
\end{align*}
\qed

\medskip

In order to calculate the volume $\vol(\F_{G}(\a))$ we need to count
the number of times leaves of the form $G(\m)$ appear in a certain reduction tree
$R^{\leftarrow}_G$ and sum over all their volumes. We tackle this in the next section. 

\section{The canonical subdivision of $\F_G(\a)$ \\ aka proving  the Lidskii volume formula}
\label{sec:thm1}

This section is devoted to proving the Lidskii volume formula \eqref{eq:vol}. We achieve this by constructing a canonical subdivision of  $\F_G(\a)$ via the compounded subdivision lemma.  In the canonical subdivision we know the volume of each of the full dimensional polytopes (Lemma \ref{lem:Gm})  -- referred to as \textbf{cells} of the   subdivision -- and we   count how many of each of the
cells occur in the canonical subdivision.

\subsection{The canonical compounded reduction tree} \label{ccrt}Given $\F_G(\a)$, $\aa=(a_1,\ldots,a_n,-\sum_{i=1}^n a_i)$, $a_i \in \mathbb{Z}_{\geq 0}$, 
let  ${R^{\leftarrow}_G}$ be the compounded  reduction tree obtained by executing the compounded reductions described in the compounded subdivision lemma    on vertices $n, n-1, \ldots, 2$ of $G$  in this order. We refer to   ${R^{\leftarrow}_G}$ as the \textbf{canonical compounded reduction tree} of $G$, or CCRT for short. 
Figure \ref{fig:right_to_left} shows an example of one path from $G$ to a full dimensional leaf in ${R}^{\leftarrow}_G$. 

We refer to the subdivision obtained from the CCRT via  the    compounded subdivision lemma as the \textbf{canonical subdivision} of $\F_G(\a)$. See Figure \ref{fig:subdivBVpoly} for an example. We note that the compounded subdivision lemma implies that the canonical subdivision is a dissection; the results of \cite[Section 6] {jessica} imply that it is also a subdivision.

\begin{figure}
\begin{center}
\includegraphics[scale=.6]{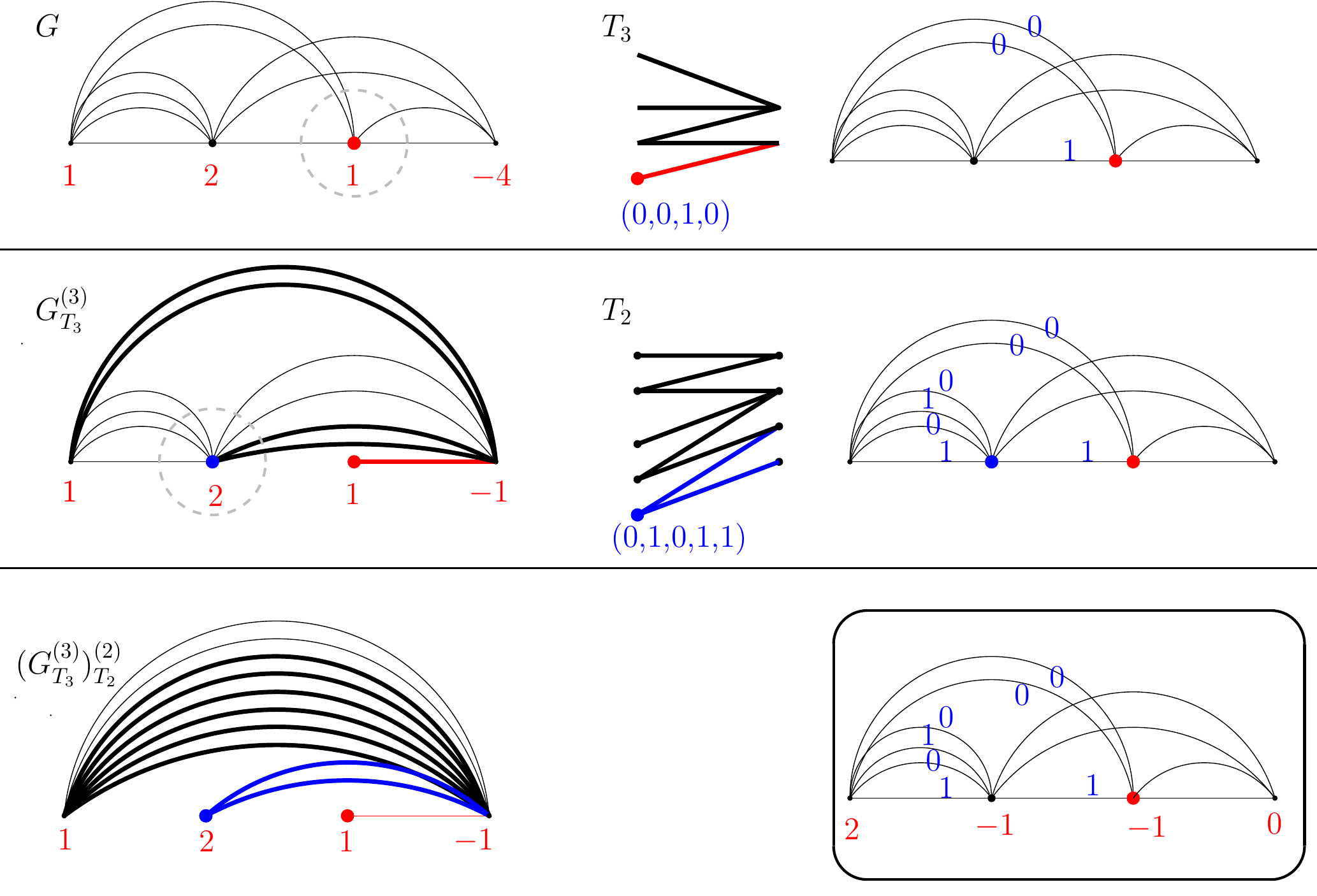}
\caption{Example of a path in the CCRT
  $R^{\leftarrow}_G$. The graph $G$ is the top left graph. The bottom
  left graph, $G(8,2,1)$, is a leaf of the CCRT. The
 compounded  reductions performed in order to arrive to this copy of $G(8,2,1)$
  are encoded by the trees/compositions $T_3$ (at vertex $3$) and
  $T_2$ (at vertex $2$). The correspondence $\Phi$ in the proof of Lemma~\ref{lemma:ng} encodes this path as the integral flow on $G$ at the
  bottom right.}
\label{fig:right_to_left}
\end{center}
\end{figure}

\subsection{Encoding the leaves of the CCRT}  
By Lemma \ref{lem:leaves} only the graphs $G(\m)$ appear as leaves of the CCRT ${R^{\leftarrow}_G}$. 
Let $N^{\leftarrow}_G(\m)$ be the number of times the 
leaf $G(\m)$ appears in  ${R^{\leftarrow}_G}$. The next key lemma shows that this number is given by a value of the Kostant partition function. This result is a generalization of \cite[Thm. 6.1]{MM}.

\begin{lemma} \label{lemma:ng} Let  $G(\m)$ be a full dimensional leaf
  of  the reduction tree ${R^{\leftarrow}_G}$ of $\F_G(\a)$.  Then the number of times the leaf $G(\m)$ appears in $R_G^{\leftarrow}$ is  
  \be N^{\leftarrow}_G(\m)=K_G(m_1-\out_1, m_2-\out_2,\ldots,
  m_n-\out_n, 0),\eee where $\out_i$ is the outdegree of vertex $i$ in $G$. 

\end{lemma}

The proof of this lemma will use the following result about the edges of the graphs $G_T^{(i)}$ appearing in $R_G^{\leftarrow}$.

\begin{proposition} \label{prop:extra-edges}
Given graphs $G$ and $G^{(i)}_T$ as above with $a_i\geq0$ and $k<i$ we have that 
\begin{compactitem}
\item[(i)] the incoming edges $\mathcal{I}_k(G^{(i)}_T)$ and
  $\mathcal{I}_k(G)$ are equal, 
\item[(ii)] if $T$ is given by the composition $(b_e,m_i-1)_{e\in \mathcal{I}_i(G)}$, then $G^{(i)}_T$ has $b_e+1$ edges $edge(\cdot,e)$ one of which corresponds to the original edge $e$ in $G$ and $b_e$ extra edges.
\end{compactitem}
\end{proposition}

\begin{proof}
This follows from the construction of $G_T^{(i)}$. 
\end{proof}

\begin{example}
In Figure~\ref{fig:right_to_left} the graph $G_{T_3}^{(3)}$ has
the same incoming edges to vertex $1$ as graph $G$. The tree $T_3$ is
given by the composition $(0,0,1,0)$. Since in this composition
$b_{(2,3)}=1$ then  $G_{T_3}^{(3)}$ has two edges of the form
$edge(\cdot,(2,3))$, which are the two copies of $(2,4)$. 
\end{example}

\begin{proof}[Proof of Lemma~\ref{lemma:ng}] In ${R^{\leftarrow}_G}$
  consider a path from $G$ to a leaf. It is obtained by picking
  particular trees $T_n, \ldots, T_2$ at the vertices $n, n-1, \ldots,
  2$ during a  compounded  reduction.  We denote the resulting graphs
  by $G_n, G_{n-1},\ldots, G_2$, respectively. That is, $G_i =
  (G_{i+1})^{(i)}_{T_i}$ where $T_i$ is the noncrossing tree encoding
  the subdivision on vertex $i$ and $G_{n+1}:=G$. 

The number $N^{\leftarrow}_G(\m)$ equals the number of tuples of
noncrossing trees ${\bf T} :=(T_2,\ldots,T_n)$ where the tree $T_i$ is
such that $G_i = (G_{i+1}^{(i)})_{T_i}$ and $\deg_{T_i}(i) = m_i$. We
give a correspondence between tuples ${\bf T}$ and integral flows
$\mathbf{f}_G$ on $G$ with netflow 
\[
{\bf a}({\bf m}):=(m_1-\out_1,\ldots,m_n-\out_n,0),
\]
where $m_1 = \#E(G)-\sum_{i=2}^n m_i$.

For $i=n,n-1,\ldots,2$, 
by Proposition~\ref{prop:extra-edges}(i) we
have that $\mathcal{I}_i(G_{i+1})=\mathcal{I}_i(G)$, thus we can encode the
tree $T_i$ as the composition of $\#\mathcal{O}_i(G_{i+1})-1$ of the
form $(b_e,m_i-1)_{e\in \mathcal{I}_i(G)}$. With this setup set $f(e)
= b_e$, and set zero flow $f((\cdot,n+1))=0$ on the incoming edges to
vertex $n+1$. This defines an integral flow ${\bf f}_G$ on $G$. Finally, set $\Phi({\bf T}) = {\bf f}_G$. For an example of $\Phi$, see Figure \ref{fig:right_to_left}.

Next, we calculate the netflow of the integral flow ${\bf f}_G$. For each $i=2,\ldots,n$, by construction of ${\bf f}_G$ we have that  
\begin{equation} \label{eq:pre_netflow}
\sum_{e \in \mathcal{I}_i(G)} f(e)  =  \out_i(G_{i+1}) - m_i. 
\end{equation}

By Proposition~\ref{prop:extra-edges}(ii), the outgoing edges of
vertex $i$ in $G_{i+1}$ correspond to the original outgoing edges in
$\mathcal{O}_i(G)$ and extra $b_{(i,j)}$ edges coming from the
composition corresponding to the tree $T_j$ and edge $(i,j)$ in
$\mathcal{I}_j(G_{j+1})=\mathcal{I}_j(G)$. Since this edge $(i,j)$ of
$G_{j+1}$ is also an edge in $\mathcal{O}_i(G)$ then we have that 
\begin{equation} \label{eq:post_netflow}
\out_i(G_{i+1}) = \out_i(G) + \sum_{e \in \mathcal{O}_i(G)} f(e).
\end{equation}
Combining \eqref{eq:pre_netflow} and \eqref{eq:post_netflow} we obtain that the netflow of vertex $i$ in ${\bf f}_G$ is $m_i-\out_i(G)$. Next we calculate  the netflow on vertex $1$. Since $G_2 = G(\m)$ then $\out_1(G_2) = m_1$. Also, by the previous argument \eqref{eq:post_netflow} holds for $i=1$, thus 
\[
\sum_{e \in \mathcal{O}_1(G_2)} f(e) = m_1-\out_1(G),
\]
as desired.

Next we show that $\Phi$ is a bijection by building its inverse. Given a flow ${\bf f}_G$ with netflow ${\bf a}({\bf m})$, we read off the flows on the edges $\mathcal{I}_i(G)$ for $i=2,\ldots,n$ to obtain compositions of $\out_i(G)-1+\sum_{e \in \mathcal{O}_i(G)} b(e)$ of the form $(b_e,m_i-1)_{e\in \mathcal{I}_i(G)}$ if $m_i>0$ or of the form $(b_e)_{e\in \mathcal{I}_i(G)}$ if $m_i=0$. We encode these compositions as bipartite noncrossing trees $T_2,\ldots,T_n$. By construction and \eqref{eq:post_netflow}, the number of outgoing vertices of $T_i$ is $\out_i(G_i)$. We set $\Psi({\bf f}_G) = (T_2,\ldots,T_n)$. By construction one can show that $\Psi=\Phi^{-1}$, thus $\Phi$ is a bijection. This shows that $N^{\leftarrow}_G(\m)$ equals the number of integral flows on $G$ with netflow ${\bf a}({\bf m})$.
\end{proof}

\subsection{Which $G(\m)$ appear as  leaves in the CCRT} \label{subsec:chardominance}

The next result characterizes the vectors $\m$ encoding the full dimensional leaves of the reduction tree $R_G^{\leftarrow}$ of the flow polytope $\F_G(\a)$.

\begin{theorem} \label{thm:leaves-dominance-order}
Given a flow polytope $\F_G(\a)$, $\aa=(a_1,\ldots,a_n,-\sum_{i=1}^n
a_i)$, $a_i \in \mathbb{Z}_{\geq 0}$, the graph $G(\m)$ is a full
dimensional leaf of the CCRT $R^{\leftarrow}_G$   if and only if
$\m=(m_1,\ldots,m_n)$ is a composition of $\#E(G)$   and
$(m_1,\ldots,m_n) \geq (\out_1,\ldots,\out_n)$ in dominance order.  
\end{theorem}

This result is proved via two lemmas.

\begin{lemma} \label{lem:dominanceorder}
Let  $G(\m)$ be a full dimensional leaf
  of  the CCRT  ${R^{\leftarrow}_G}$. Then
  $(m_1,\ldots,m_n) \geq (\out_1,\ldots,\out_n)$ in dominance order.
\end{lemma}

\begin{lemma} \label{lem:dominanceorder-converse}
If  $\m=(m_1,\ldots,m_n)$ is a  composition of $\#E(G)$ with
$(m_1,\ldots,m_n) \geq (\out_1,\ldots,\out_n)$ in dominance order,
then  the CCRT $R_G^{\leftarrow}$ has full dimensional leaves
$G(\m)$. 
\end{lemma}

\begin{proof}[Proof of Theorem~\ref{thm:leaves-dominance-order}]
The characterization follows by Lemmas~\ref{lem:dominanceorder} and \ref{lem:dominanceorder-converse}.
\end{proof}

The rest of this subsection is devoted to the proofs of the two lemmas.

\begin{proof}[Proof of Lemma~\ref{lem:dominanceorder}]
By Lemma~\ref{lem:leaves} we know that $m_1+\cdots + m_n =
\out_1+\cdots+ \out_n$. Since these sums are equal, showing  $(m_1,\ldots,m_n) \geq
(\out_1,\ldots,\out_n)$ is equivalent to showing  $(m_n,\ldots,m_1) \leq
(\out_n,\ldots,\out_1)$. We show the latter by induction on the number of vertices of $G$ with incoming edges. 

We first show that $m_n \leq \out_n$. The first reduction in $R_G^{\leftarrow}$ occurs at
vertex $n$ of $G$ and yields a graph $G^{(n)}_{T}$ with no incoming
edges to vertex $n$. If $a_n=0$ then $m_n=1$ and so the inequality
holds (since we require ${\rm outd}_i\geq 1$ for all $i \in [n]$). If $a_n>0$ then the tree $T$ has left vertices $\mathcal{I}_n\cup
  \{v_n\}$ and right vertices $\mathcal{O}_n$ with
  $\deg_T(v_n)=m_n$. Thus $m_n \leq \#\mathcal{O}_n = \out_n$. Also
compared to $G$, the graph $G^{(n)}_T$ has $\out_n-m_n$ new edges
  $(i,n+1)$ for $i<n$. Thus
\[
  (\out'_1 + \cdots + \out'_{n-1}) - (\out_1+\cdots + \out_{n-1}) = \out_n-m_n,
\]
where $\out'_i$ is the outdegree of vertex $i$ in $G^{(n)}_T$.  So for $k=1,\ldots,n-2$ we have 
\begin{equation} \label{eq:dom-order}
\out'_{n-1} + \out'_{n-2} + \cdots + \out'_{n-k} \leq \left(\out_{n-1}
+ \cdots + \out_{n-k}\right) + \out_n - m_n.
\end{equation}
If $G(m_1,\ldots,m_n)$ is a full dimensional leaf of
${R^{\leftarrow}_G}$ then $G^{(n)}_T(m_1,\ldots,m_{n-1})$ is a full
dimensional leaf of
the reduction tree ${R^{\leftarrow}_{G^{(n)}_T}}$. By induction we have $(m_{n-1},\ldots,m_2,m_1) \leq (\out'_{n-1},\ldots,\out'_1)$. This
combined with \eqref{eq:dom-order} gives
\begin{align*}
m_{n-1} + m_{n-2} + \cdots + m_{n-k} &\leq \out'_{n-1} + \out'_{n-2} +
\cdots + \out'_{n-k} \\
&\leq \out_{n} + \out_{n-1}+\cdots + \out_{n-k} - m_n.
\end{align*}
Thus $(m_n,\ldots,m_1) \leq (\out_n,\ldots,\out_1)$ as desired.
\end{proof}

We now prove the converse of the previous lemma.

\begin{proof}[Proof of Lemma \ref{lem:dominanceorder-converse}] 
Since $m_1+\cdots + m_n = \out_1+\cdots + \out_n$, then
$(m_1,\ldots,m_n) \geq (\out_1,\ldots,\out_n)$ is equivalent to
$(m_n,\ldots,m_1) \leq (\out_n,\ldots,\out_1)$. We show the result by
induction on the number of vertices of $G$ with incoming edges.  Let
$T$ be the tree encoded by the composition  $(0^{\ind_n-1},\out_n-m_n
,m_n-1)$. By Lemma~\ref{lem:big_red} the graph $G_T^{(n)}$ is a node
of the reduction tree $R_G^{\leftarrow}$. This graph has
$\#E(G)-n-m_n$ edges, no incoming edges to vertex $n$ and if $\out'_i$ is the outdegree of vertex $i$ in $G_T^{(n)}$ then 
\begin{equation} \label{eq:converse-outv}
\out'_{n-1} =  
\out_{n-1} + \out_n-m_n. 
\end{equation}
Now, the weak composition $(m_{n-1},\ldots,m_1)$ of $\#E(G)-n-m_n$  is $\leq (\out'_{n-1},\ldots,\out'_1)$ in dominance order since by \eqref{eq:converse-outv} 
\[
\out'_{n-1} + \cdots + \out'_{n-k} = \out_{n-1} + \cdots + \out_{n-k} + \out_n - m_n \geq m_{n-1} + \cdots + m_{n-k}.
\]
By induction $G_T^{(n)}(m_1,\ldots,m_{n-1})=G(\m)$ is a full dimensional leaf of the reduction tree of $G_T^{(n)}$. Since $G_T^{(n)}$ is a node of the reduction tree of $G$ then $G(\m)$ is a full dimensional leaf of the reduction tree $R_G^{\leftarrow}$ as desired.
\end{proof}

\subsection{Computing the volume  of $\F_G(\a)$} To finish the proof of the Lidskii volume formula 
\eqref{eq:vol} we fix the reduction tree $R_G^{\leftarrow}$ to
subdivide $\F_{G}(\a)$ into full dimensional leaves $\F_{G({\bf
    m})}(\a)$. Then
\[
\vol(\F_G(\a)) = \sum_{ {\bf m}} 
\vol(\F_{G({\bf m})}(\a)) \cdot N^{\leftarrow}_G.
\]
The relation \eqref{eq:vol} then follows by using Lemma~\ref{lem:Gm} to compute
$\vol(\F_{G(\m)}(\a))$, and using Lemma~\ref{lemma:ng} to compute
$N^{\leftarrow}_G$, and relabeling $m_i$ to $j_i+1$. The compositions ${\bf j}$
add up to $\om_1+\cdots + \om_n =m-n$ and they are exactly those that are  $\geq
(\om_1,\ldots,\om_n)$ in dominance order by Theorem~\ref{thm:leaves-dominance-order}.

\begin{example} \label{ex:k4vol}
The reduction tree in Figure~\ref{fig:complete} is in fact
$R_{k_4}^{\leftarrow}$. Since for $k_4$ we have that $(\om_1,\om_2,\om_3)=(2,1,0)$, the Lidskii formula \eqref{eq:vol} gives
\[
\vol\F_{k_4}({\bf 1}) = \binom{3}{2,1,0} K_{k_4}(2-2,1-1,0-0,0)
+ \binom{3}{3,0,0}K_{k_4}(3-2,0-1,0-0,0) = 3\cdot 1 + 1\cdot 1  =4.
\]
This corresponds to a subdivision of the polytope $\F_{k_4}(1,1,1,-3)$
by the plane $f_{12}=f_{24}$ as indicated by the reduction tree in
Figure~\ref{fig:complete}. See Figure~\ref{fig:subdivBVpoly} for an
illustration of this subdivision. For more examples, see Appendix~\ref{sec:appendix}.
\end{example}

\begin{figure}
    \centering
\includegraphics[scale=0.8]{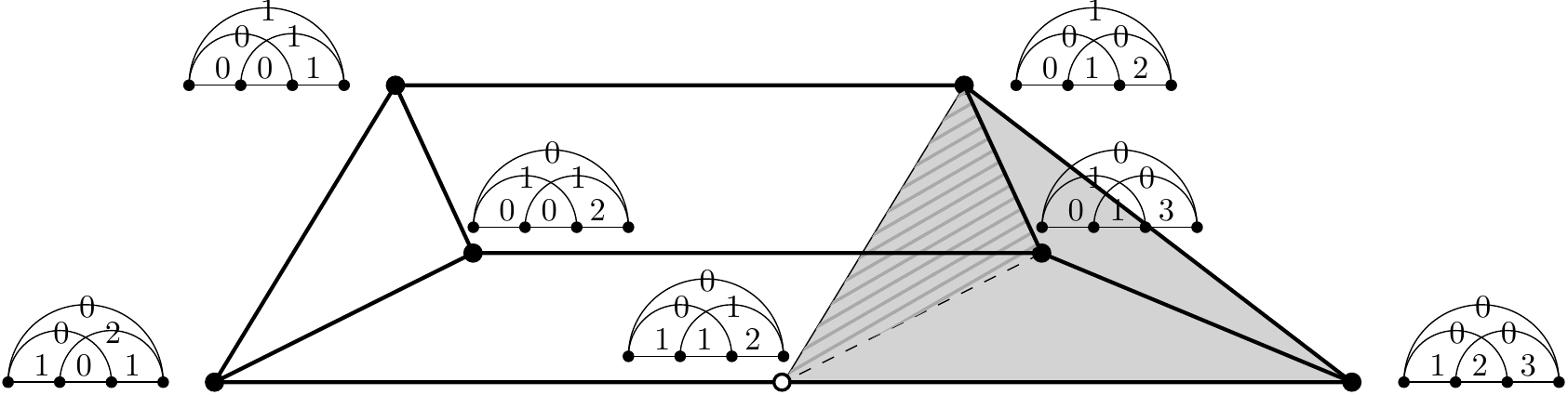}
    \caption{Canonical subdivision of the polytope $\F_{k_4}(1,1,1,-3)$ with volume $4$ and $7$ lattice points.}
    \label{fig:subdivBVpoly}
\end{figure}

\subsection{Alternative volume formula in terms of indegrees}

In this section we apply the symmetry of the Kostant partition function
(Corollary~\ref{cor:symkpg}) to prove
Corollary~\ref{cor:vol_indegree} which gives an indegree formula for
the volume of flow polytopes.  

\begin{proof}[Proof of Corollary~\ref{cor:vol_indegree}]
Using Proposition~\ref{prop:symfp} and \eqref{eq:vol} we get
\begin{align*}
\vol \F_G(\sum_{i=1}^nb_i,-b_1,\ldots,-b_n) &=
                                   \vol \F_{G^r}(b_n,\ldots,b_1,-\sum_{i=1}^n
                                    b_i)\\
&= \sum_{{\bf j}} \binom{m-k-n}{j_1,\ldots,j_n} b_n^{j_n}\cdots
  b_1^{j_1} K_{G^r}(j_n-\out_1^{G^r}+1,\ldots,j_1-\out_n^{G^r}+1,0).
\end{align*}
Using Corollary~\ref{cor:symkpg}   on the RHS above we get:
\begin{align*}
\vol \F_G(\sum_{i=1}^nb_i,-b_1,\ldots)  &= \sum_{{\bf j}} \binom{m-k-n}{j_1,\ldots,j_n} b_1^{j_n}\cdots
  b_n^{j_1} K_{G}(0, \out_n^{G^r} - j_1-1,\ldots, \out_1^{G^r} -j_n-1)\\
&= \sum_{{\bf j}} \binom{m-k-n}{j_1,\ldots,j_n} b_1^{j_1}\cdots
  b_n^{j_n} K_{G}(0, \ind_2^G - j_1-1,\ldots, \ind_{n+1}^{G} -j_n-1),\\
\end{align*}
where the last equality follows since the outdegree of vertex $i$ in $G^r$ equals the indegree of
vertex $n+2-i$ in $G$. 
\end{proof}

\section{Proof of the Lidskii formulas for lattice points}
\label{sec:kpf}

In this section we prove the Lidskii formulas \eqref{eq:kost} and \eqref{eq:kostmultiset} for the number of lattice points of flow polytopes. 
The key to our  combinatorial proof of \eqref{eq:kost} lies in comparing the basic and compounded reduction trees of the  graph $G$, as we do below.   

\subsection{The basic reduction tree revisited} 

  There are two important properties of  a  BRT: 
\begin{compactitem}
\item[1.]By Proposition~\ref{red0} we get
a subdivision of the original flow polytope from the leaves of a BRT.
\item[2.] Unlike in a CRT, in a BRT  we obtain
leaves that are not necessarily full dimensional. 
\end{compactitem}

\medskip

The following lemma is implicit in \cite[\S 5.3]{MM}:

\bl \label{lem:cb} Given the  CCRT for a graph $G$ on the vertex set $[n+1]$,  there is a BRT whose full dimensional leaves coincide with those of the CCRT. \el

 \proof \textit{(Sketch)} Construct the desired BRT by doing basic reductions on vertices $n, \ldots, 2$ in this order. At each vertex $i$ repeatedly do BR on the longest possible edges available, until there are still edges on which the BR can be performed.  (The length of an edge $(i,j)$ is $j-i$.) When there are no more edges proceed the same way at vertex $i-1$.
 \qed

\begin{example}
 Figure \ref{fig:brt} has an
example of a BRT where the full dimensional leaves are boxed. Note that
  in Figure \ref{fig:complete} we subdivided the same flow polytope
  with the CCRT and got the same leaves as the full dimensional leaves
  in this BRT. 
\end{example}

\begin{figure}
\centering
\includegraphics[scale=0.75]{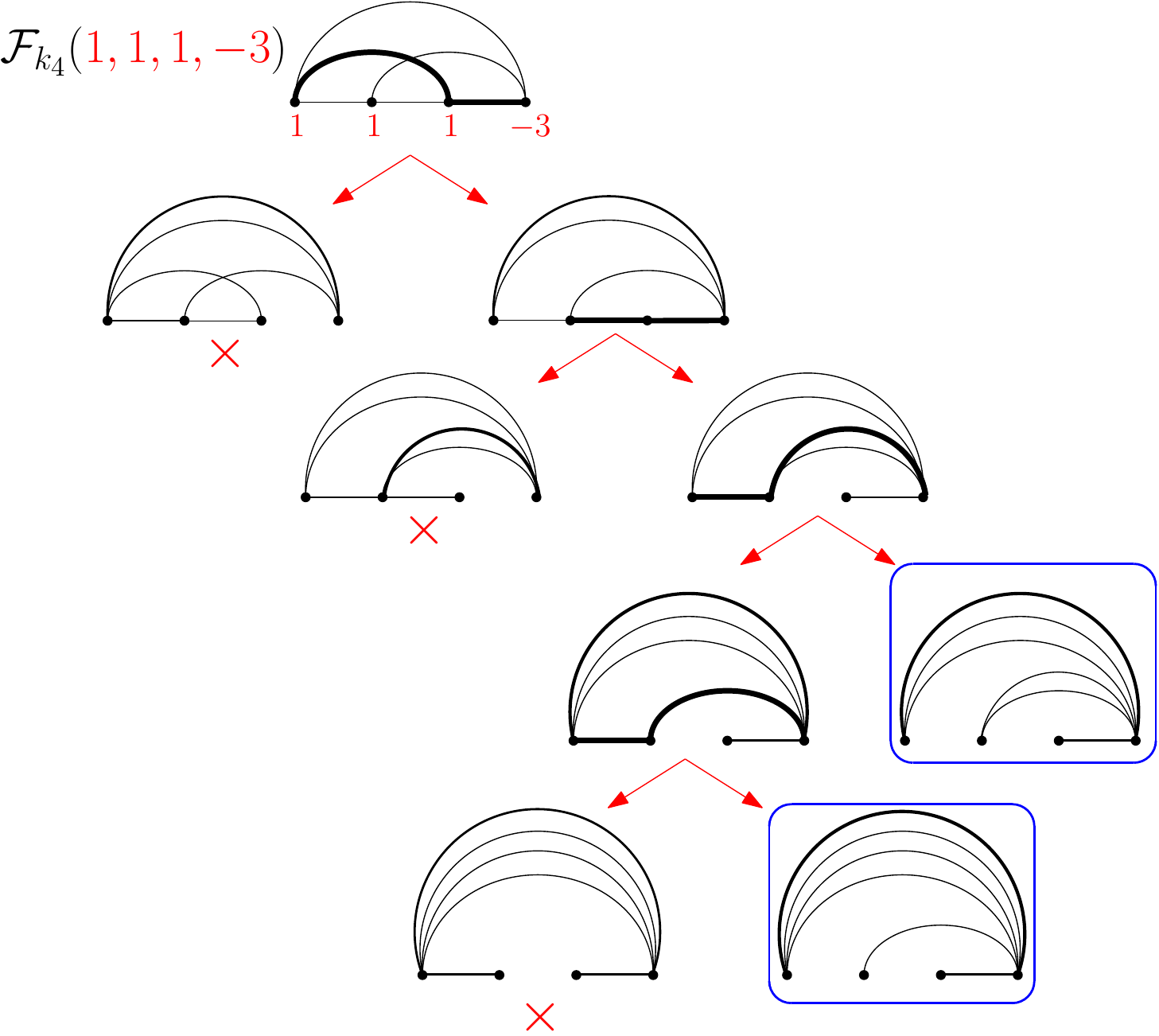}
\caption{Basic reduction tree (BRT) for the polytope
  $\F_{k_4}(1,1,1-3)$. The full-dimensional leaves are boxed while the
  lower dimensional are marked by $\textcolor{red}{\times}$. Note that
  in Figure \ref{fig:complete} we subdivided the same flow polytope
  with the CRT and got the same leaves as the full dimensional leaves in this BRT demonstrating Lemma \ref{lem:cb}.}
 \label{fig:brt}
\end{figure}

\subsection{Encoding the leaves of the BRT for the Kostant partition function}

By \eqref{gsKA} the function $K_G({\bf a})$ is obtained by the
following sums of coefficient extractions
\begin{equation} \label{k}
K_G({\bf a})=[{\bf x}^{\bf a}] {\prod_{(i,j)\in E(G)}} (1-x_ix_j^{-1})^{-1},
\end{equation}
where ${\bf x}^{\bf a} = x_1^{a_1}\cdots x_{n+1}^{a_{n+1}}$. The advantage of considering the BRT for obtaining \eqref{eq:kost} is that the reduction rule \eqref{R1} can easily be encoded with variables as follows. 
\begin{equation} \label{rel} \frac{1}{1-x_ax_i^{-1}}\frac{1}{1-x_i x_{b}^{-1}}= \frac{1}{1-x_a x_{b}^{-1}}\left(\frac{x_a x_i^{-1}}{1-x_a x_{i}^{-1}}+\frac{1}{1-x_i x_{b}^{-1}}\right).\end{equation}

We fix a BRT $R_G$ whose full dimensional leaves coincide with those
of CCRT  $R_G^{\leftarrow}$. When executing a BR on a graph $G$ as
defined in Section \ref{subsec:basicsubdiv}, we draw the BRT by letting
$G_1$ be the left child and $G_2$ be the right child of $G$. We assign
the monomial $x_a x_i^{-1}$ to the ``left" edge connecting $G$ and
$G_1$ and the constant $1$ to the ``right"edge connecting  $G$ and
$G_2$. We assign each node  $H$ of the BRT the monomial ${\bf x}^ {\bf
  H}$ obtained by multiplying the  monomials assigned to the left
edges on the unique path from the root of the BRT to $H$. We then have
the following expression for $K_G(\a)$.
 
 \begin{equation} \label{eq:ptnleaves}
K_G(\a) = \sum_H [{\bf x}^{\a}] \,\, {\bf x}^{\bf H} \prod_{(i,j) \in
  E(H)} (1-x_ix_j^{-1})^{-1} = \sum_{H} K_{H}({\bf a}-{\bf H}),
\end{equation}
where the sum is over the leaves $H$ of the BRT.

\begin{proposition} \label{prop:mon4H}
The monomial  ${\bf x}^ {\bf H}$ associated to a leaf $H$ of the BRT $R_G$ equals 
\begin{equation}
\label{eq:mon4H}
{\bf x}^{\bf H} = \prod_{i=1}^n x_i^{\out_i(H) - \out_i},
\end{equation}
where $\out_i(H)$ is the outdegree of vertex $i$ in $H$. 
\end{proposition}

\begin{proof}
 At each left step of the BRT involving a
reduction on a vertex $i$, an extra edge $(a,b)$ is added that is
outgoing with respect to $a$, an outgoing
edge $(i,b)$ is removed from the graph, and we record the remaining incoming edge
$(a,i)$ in the numerator as $x_a x_i^{-1}$. This monomial in the numerator records
adding an outgoing edge to $a$ and removing an outgoing edge to $i$. Thus the power of $x_i$ in
the  monomial ${\bf x}^{\bf H}$ is the number of extra outgoing edges
$(i,\cdot)$ in $H$. This number equals $\out_i(H) -
\out_i(G)$. 
\end{proof}

By Lemmas~\ref{lem:leaves} and \ref{lem:cb}, the full dimensional leaves of the BRT $R_G$ are the
graphs $G(\m)$. Next we calculate the contribution from each such leaf
in \eqref{eq:ptnleaves}.

\begin{lemma} \label{lem:kost_full_leaves}
For a full dimensional leaf $G(\m)$ of the BRT $R_G$ we have that,
\[
K_{G(\m)}({\bf a} - {\bf G(\m)}) = \binom{a_1+\out_1-1}{m_1-1}\binom{a_2+\out_2-1}{m_2-1} \cdots
\binom{a_{n}+\out_n-1}{m_n-1}.
\]
\end{lemma} 

\begin{proof}
We calculate 
\[
K_{G(\m)}({\bf a} - {\bf G(\m)})  =  [{\bf x}^{\bf a}]\, {\bf x}^{G(\m)} \prod_{(i,j) \in
  E(G(\m))} (1-x_ix_j^{-1})^{-1}.
\]
By Proposition~\ref{prop:mon4H}, the monomial for the full
dimensional leaf $G(\m)$ is $\prod_{i=1}^n x_i^{m_i - \out_i}$. Next, we do the coefficient extraction to obtain the desired formula:
\begin{align*}
[{\bf x}^{\bf a}] \prod_{i=1}^n
\frac{x_i^{m_i-\out_i}}{(1-x_ix_{n+1}^{-1})^{m_i}} &=
                                                     [x_1^{a_1-m_1+\out_1}\cdots
                                                     x_n^{a_n-m_n+\out_n}]
                                                     \prod_{i=1}^n
                                                     \frac{1}{(1-x_i)^{m_i}}\\
&= \prod_{i=1}^n  [x_i^{a_i-m_i+\out_i}] (1-x_i)^{-m_i}   = \prod_{i=1}^n \binom{a_i+\out_i-1}{m_i-1}.
\end{align*}

\end{proof}

Next, we show that the lower dimensional leaves do not contribute to \eqref{eq:ptnleaves}.

\begin{lemma} \label{lem:kost_lower_leaves}
For a lower dimensional leaf $H$  of the BRT $R_G$ we have that $K_H({\bf a} - {\bf H}) = 0$.
\end{lemma}

\begin{proof}
We calculate $[{\bf x}^{\a}] \,\, {\bf x}^H \prod_{(i,j) \in
  E(H)} (1-x_ix_j^{-1})^{-1}$ for a lower dimensional leaf $H$. By Proposition~\ref{prop:mon4H} the
monomial for such leaf $H$ is $\prod_{j=1}^n
x_j^{\out_j(H)-\out_j}$. Since the leaf $H$ is not of the form $G(\m)$
then it has a vertex $k$ with incoming edges but no outgoing edges. Thus
\begin{align*}
 [{\bf x}^{\a}] \,\, {\bf x}^H \prod_{(i,j) \in
  E(H)} (1-x_ix_j^{-1})^{-1} &= [\prod_i x_i^{a_i-\out_i(H)+\out_i}] \prod_{(i,j) \in
  E(H)} (1-x_ix_j^{-1})^{-1} \\ &= K_H(a_1-\out_1(H)+\out_1,\ldots,\,a_k+\out_k\,,\ldots).
\end{align*}
However, since vertex $k$ has no outgoing edges then there are no integral
flows in $H$ with netflow $a_k+\out_k>0$ in vertex $k$. (Recall that
the  graphs $G$ we consider have $\out_i>0$ for all $i \in [n]$.)
\end{proof}

\subsection{Counting the lattice points of $\F_G(\a)$}
We now complete the proof of the Lidskii formula \eqref{eq:kost} for $K_G({\bf a})$. 

\begin{proof}[Proof of \eqref{eq:kost}]
By Lemma \ref{lem:kost_lower_leaves}, in \eqref{eq:ptnleaves} only the full dimensional leaves
contribute 
\[
K_G(\a) = \sum_{\m}  K_{G(\m)}(a_1-m_1-\out_1,\ldots,a_n-m_n-\out_n)
          \cdot N^{\leftarrow}_G
\]
We then use Lemma~\ref{lemma:ng} to compute $N^{\leftarrow}_G$ and
Lemma~\ref{lem:kost_full_leaves} to compute $K_{G(\m)}(\cdot)$,
\begin{align*}
K_G(\a) &= \sum_{\m}  K_{G(\m)}(a_1-m_1-\out_1,\ldots,a_n-m_n-\out_n)
          \cdot K_G(f_1(\m),\ldots,f_n(\m),0)\\
&= \sum_{\m}  \binom{a_1+\out_1-1}{m_1-1} \cdots
\binom{a_{n}+\out_n-1}{m_n-1} \cdot K_G(m_1-\out_1,\ldots,m_n-\out_n,0).
\end{align*}
\end{proof}

\begin{example} \label{ex:k4kost}
Continuing with Example~\ref{ex:k4vol}, the graph $k_3$ Lidskii formula
\eqref{eq:kost} gives
\[
K_{k_4}(1,1,1,-3) = \binom{3}{2}\binom{2}{1}K_{k_4}(0,0,0,0)
+ \binom{3}{3}\binom{2}{0}K_{k_4}(1,-1,0,0) = 6+1  =7.
\]
The subdivision of  $\F_{k_4}(1,1,1,-3)$
in Figure~\ref{fig:subdivBVpoly} yields
two cells with six and four lattice points each and three lattice points
in their intersection. These three points are only counted in the
first cell. For more examples, see Appendix~\ref{sec:appendix}.
\end{example}

By applying the symmetry of the Kostant partition function we
obtain an alternative formula to \eqref{eq:kost}.

\begin{corollary} 
\label{cor:kost_indegree}
Let $G$ be a graph on the vertex set $[n+1]$ with at least one
incoming edge at vertex $i$ for $i=2,\ldots,n+1$, and let $\b = (\sum_{i=1}^n
b_i,-b_1,\ldots,-b_n)$ with $b_i\in \mathbb{Z}_{\geq 0}$,  then 
\begin{equation} \label{eq:kost1}
K_{G}({\bf \b}) = \sum_{{{\bf j}}}
\binom{b_1+\i_2}{j_1} \cdots
\binom{b_{n}+\i_{n+1}}{j_{n}} \cdot   K_{G}(0, \i_2-j_1-1, \i_3-i_2-1,\ldots, \i_{n+1} - j_n-1),
\end{equation}
where the sum is over weak compositions ${\bf
  j}=(j_1,j_2,\ldots,i_n)$ of $m-n$ that are $\leq$
$(\i_2,\ldots,\i_{n+1})$ in dominance order..
\end{corollary}

\begin{proof}
The result follows by applying Corollary~\ref{cor:symkpg} to
\eqref{eq:kost} (cf. proof of Corollary~\ref{cor:vol_indegree}).
\end{proof}

\subsection{Proof of the Lidskii formula \eqref{eq:kostmultiset} for
  lattice points}

Next we prove the alternative Lidskii formula \eqref{eq:kostmultiset}
for the Kostant partition function. We first prove the results for the case
$a_i \geq \ind_i(G)$ for $i=1,\ldots,n$ and then extend them to the range $0\leq
a_i<\ind_i(G)$ using the polynomiality property of the Kostant
partition function.

The result follows mostly the same argument that proves \eqref{eq:kost} but
 instead of \eqref{rel}, we encode the reduction rule \eqref{R1}  as 
\begin{equation} \label{relB} \frac{1}{1-x_ax_i^{-1}}\frac{1}{1-x_i
    x_{b}^{-1}}= \frac{1}{1-x_a x_{b}^{-1}}\left(\frac{1}{1-x_a
      x_{i}^{-1}}+\frac{x_ix_b^{-1}}{1-x_i
      x_{b}^{-1}}\right).\end{equation}
We fix a BRT $R_G$ whose full dimensional leaves coincide with those
of CCRT $R_G^{\leftarrow}$. When executing a BR on a graph $G$ as
defined in Section \ref{subsec:basicsubdiv}, we draw the BRT by having
$G_1$ be the left child and $G_2$ be the right child of $G$. We assign the constant $1$ to the ``right"edge connecting  $G$ and
$G_1$, and the monomial $x_i x_b^{-1}$ to the ``right" edge connecting $G$ and
$G_2$.
Then the analogue of \eqref{eq:ptnleaves} is 
\begin{equation} \label{eq:ptnleavesmultiset}
K_G(\a) = \sum_H   [{\bf x}^{\bf a}] \,{\bf x}^{\bf H'} \prod_{(i,j)
  \in E(H)} (1-x_ix_j^{-1}){-1} = \sum_H K_H({\bf a} - {\bf H}'),
\end{equation}
where the sum is over all leaves of the BRT $R_G$. 

\begin{proposition} \label{prop:mon4Hmultiset}
The monomial ${\bf x}^{\bf H'}$ associated to a leaf $H$ of the BRT
$R_G$ equals
\[
{\bf x}^{\bf H'} = \prod_{j=2}^{n+1} x_j^{\ind_j- \ind_j(H)}.
\]
\end{proposition}

\begin{proof}[Proof sketch]
This monomial
comes from the {\em right} steps in the reduction tree, where an
incoming edge $(a,i)$ is removed and we record the outgoing edge
$(i,b)$ in the numerator as $x_ix_b^{-1}$.
\end{proof}

As in the proof of \eqref{eq:kost}, the full dimensional leaves of
the BRT $R_G$ are the graphs $G({\bf m})$. Next we calculate the
contribution from each such leaf in \eqref{eq:ptnleavesmultiset}.

\begin{lemma} \label{lem:kost_full_leaves_multiset}
Let $\a = (a_1,\ldots,a_n, -\sum_{i=1}^n a_i)$ with $a_i \in
\mathbb{Z}_{\geq 0}$ with $a_i \geq \ind_i(G)$. For a full dimensional leaf $G({\bf m})$ of the BRT $R_G$ we have that
\[
K_{G({\bf m})}({\bf a} - {\bf G({\bf m})'}) =  \multiset{a_1-\ind_1+1}{m_1-1}\cdots
\multiset{a_{n}-\ind_n+1}{m_n-1}.
\] 
\end{lemma}

\begin{proof}
By Proposition~\ref{prop:mon4Hmultiset} the monomials for each full dimensional leaf $G({\bf m})$ are the
same: 
\[
x_{n+1}^{m}\prod_{j=1}^{n+1} x_j^{\ind_j(G)}.
\]
(For convenience, we included superflously the variable $x_1$ since $\ind_1(G)=0$.) Thus
\begin{align*}
[{\bf x}^{\bf a}] \prod_{j=1}^n
  \frac{x_j^{\ind_j}}{(1-x_jx_{n+1}^{-1})^{m_j}} = \prod_{j=1}^n \left(
  [x_j^{a_j-\ind_j}] (1-x_j)^{-m_j} \right) = \prod_{j=1}^n
  \binom{a_j-\ind_j + m_j-1}{m_j-1},
\end{align*}
where we used the assumption that $a_j - \ind_j \geq 0$ for $j=1,\ldots,n$.
\end{proof}

Next, we show that the lower dimensional leaves do not contribute to \eqref{eq:ptnleavesmultiset}.

\begin{lemma}  \label{lem:kost_lower_leaves_multiset}
Let $\a = (a_1,\ldots,a_n, -\sum_{i=1}^n a_i)$ with $a_i \in
\mathbb{Z}_{\geq 0}$ with $a_i \geq \ind_i(G)$. For  a lower
dimensional leaf $H$ of the BRT $R_G$ we have that $K_H({\bf a} - {\bf
  H}') =0$. 
\end{lemma}

\begin{proof}
By Proposition~\ref{prop:mon4Hmultiset} the monomial for such leaf $H$ is $\prod_{j=2}^{n+1} x_j^{\ind_j -
  \ind_j(H)}$. Since the leaf $H$ is not of the form $G({\bf m})$ then
it has a vertex $k\geq 2$ with incoming but no outgoing edges. Thus
\[
[{\bf x}^{\bf a}] \,\,\, {\bf x}^{\bf H'} \prod_{(i,j) \in E(H)}
(1-x_ix_j^{-1})^{-1} =
K_H(a_1,a_2-\ind_2+\ind_2(H),\ldots,a_k-\ind_k+\ind_k(H),\ldots).
\]
However, since vertex $k$ has no outgoing edges then there are no
integral flows in $H$ with netflow $a_k -\ind_k + \ind_k(H) >0$ at
this vertex.
\end{proof}

\begin{proof}[Proof of \eqref{eq:kostmultiset}]
We start by assuming that $a_i \geq \ind_i(G)$. By Lemma~\ref{lem:kost_lower_leaves_multiset}, in
\eqref{eq:ptnleavesmultiset} only the full dimensional leaves contribute
\[
K_G(\a) = \sum_{\bf m} K_{G(\m)}(a_1-\ind_1,\ldots,a_n-\ind_n) \cdot N_G^{\leftarrow}.
\]
We then use Lemma~\ref{lemma:ng} to compute $N^{\leftarrow}_G$ and
Lemma~\ref{lem:kost_full_leaves_multiset} to compute
$K_{G(\m)}(\cdot)$,
\begin{align*}
K_G(\a) &= \sum_{\m}  K_{G(\m)}(a_1-\ind_1,\ldots,a_n-\ind_n) \cdot
          K_G(m_1-\out_1,\ldots,m_n-\out_n,0),\\
&=  \sum_{\m} \multiset{a_1-\ind_1+1}{m_1-1}\cdots
\multiset{a_{n}-\ind_n+1}{m_n-1} \cdot K_G(m_1-\out_1,\ldots,m_n-\out_n,0).
\end{align*}
Finally, to extend the identity to the cases where $0\leq a_i<\ind_i(G)$ we
use  the polynomiality property of $K_G({\bf a})$ (Proposition~\ref{prop:polynomiality4kpf}).
\end{proof}

\begin{example}
To contrast Example~\ref{ex:k4kost}, for the graph $k_3$ we have
$(\i_1,\i_2,\i_3) = (-1,0,1)$, so the alternative Lidskii formula
\eqref{eq:kostmultiset}, gives
\[
K_{k_4}(1,1,1,-3) = \binom{3}{2}\binom{1}{1}K_{k_4}(0,0,0,0)
+ \binom{4}{3}\binom{0}{0}K_{k_4}(1,-1,0,0) = 3+4 =7.
\]
The subdivision of  $\F_{k_4}(1,1,1,-3)$
in Figure~\ref{fig:subdivBVpoly} yields
two cells with six and four lattice points each and three lattice points
in their intersection. In contrast with Example~\ref{ex:k4kost}, these
three points are now counted in the second cell. For more examples, see Appendix~\ref{sec:appendix}.
\end{example}

\section{Enumerative properties of the canonical subdivision and Lidskii formulas} \label{sec:numpieces}

In this section we give enumerative properties of the Lidskii formulas and of the canonical subdivision of flow polytopes $\F_{G}(\a)$ we used to prove Theorem \ref{thmA}. We illustrate the results with the Stanley--Pitman polytope ($G = \Pi_n$), the Baldoni--Vergne polytope ($G = k_{n+1}$), and a generalization of the former (see Section~\ref{sec:genPS}).

\subsection{Number of types of cells in the  subdivision}
Recall that we call cells the full dimensional polytopes in the
canonical subdivision of $\F_{G}(\a)$. In this section we assume $a_i
\in \Z_{> 0}$ so that the cells are present. Moreover, two cells are said to be of the same type if they are integrally equivalent.

\begin{theorem} \label{cor:corrlatptsPS}
The types of cells of the canonical subdivision of
$\F_G(a_1,a_2,\ldots,a_n,-\sum_i a_i)$
are in one-to-one correspondence with lattice points of $\PS(\om_n,\om_{n-1},\ldots,\om_2)$.
\end{theorem}

\begin{proof}
The cells of the canonical subdivision of $\F_{G}(\a)$ are characterized by
tuples $(j_1,\ldots,j_n)$ of nonnegative integers satisfying
\begin{align*}
j_1+\dots + j_k &\geq \om_1+\cdots + \om_k, \text{ for } k=1,\ldots,n-1\\
j_1+\cdots+j_n &=\om_1+\cdots
+ \om_n. 
\end{align*}
These conditions are equivalent to 
\begin{align*}
j_{n} + j_{n-1} +\dots + j_{n-k+1} &\leq \om_n+\om_{n-1}+\cdots + \om_{n-k+1}, \text{ for } k=1,\ldots,n-1\\
j_1+\cdots+j_n &=\om_1+\cdots
+ \om_n, 
\end{align*}
which in turn is equivalent to the tuple $(j_n,j_{n-1},\ldots,j_2)$ being a lattice point of the Pitman--Stanley polytope
$\PS(\om_{n},\om_{n-1},\ldots,\om_2)$ and $j_1=(\om_1+\cdots+\om_n)-(j_2+\cdots+j_n)$.
\end{proof}

\begin{corollary} \label{cor:numtypescells}
The number $N$ of types of cells of the canonical subdivision of the
polytope $\F_{G}(\a)$ is the number of plane partitions of shape $(\om_n, \om_n+\om_{n-1},\ldots, \om_n + \cdots + \om_2)$ with largest part at most $2$ which is given by the following determinant
\[
N = \det \left[ \binom{\om_{i+1}+\cdots+\om_{n} + 1}{i-j+1} \right]_{1\leq i,j\leq n-1}.
\] 
\end{corollary}

\begin{proof}
The result follows by combining Theorem~\ref{cor:corrlatptsPS} with Theorem~\ref{thm:detlatptsPS}.
\end{proof}

We next apply this result to the Pitman--Stanley polytope and the Baldoni--Vergne polytope.

\begin{corollary} \label{cor:numtypesPS}
The number of types of cells of the canonical subdivision of the Pitman--Stanley polytope $\F_{\Pi_n}(\a)$ is $C_n$.
\end{corollary}

\begin{proof}
For the graph $\Pi_n$ we have that $\om_i=1$ so by Corollary~\ref{cor:numtypescells}, the number of types of cells of the canonical subdivision equals the number of plane partitions of shape $(1,2,\ldots,n-1)$ with largest part at most 2. These plane partitions are easily seen to be in bijection with Dyck paths of size $n$ (consider the interface between $1$s and $2$s in such a plane partition). 
\end{proof}

\begin{corollary}
The number $t_n$ of types of cells  of the canonical subdivision of the Baldoni--Vergne polytope $\F_{k_{n+1}}(\a)$ equals the number of plane partitions of shape $(\binom{2}{2},\binom{3}{2},\ldots,\binom{n-1}{2})$ with largest part at most $2$.The number $t_n$  is given by the determinant
\begin{equation} \label{eq:numtypescellsBV}
t_n = \det \left[ \binom{\binom{n-i}{2} + 1}{i-j+1} \right]_{1\leq i,j\leq n-1}.
\end{equation}
\end{corollary}

\begin{proof}
This is a direct application of Corollary~\ref{cor:numtypescells}. For the complete graph $k_{n+1}$ we have that $\om_i = n-i$.  
\end{proof}

\begin{example} \label{ex:numtypecells}
The subdivision of $\F_{k_4}(1,1,1,-3)$ illustrated in Figure~\ref{fig:subdivBVpoly} has $t_3=2$ types of cells. The subdivision of $\F_{k_5}(1,1,1,1,-4)$ has $t_4=7$ types of cells as can be calculated via the determinant in \eqref{eq:numtypescellsBV}. For the terms of the sequence $(t_n)_{n\geq 0}$ see \cite[\href{http://oeis.org/A107877}{A107877}]{OEIS}.
\end{example}

\subsection{Number of cells in the canonical subdivision}

Given a graph $G$ on the vertex set $[n+1]$, let $G^\star$ and $G^{\circ}$ be the graphs
obtained from $G$ by adding a vertex $0$ adjacent to vertices $1,2,\ldots,n$ and adjacent to vertices $1,2,\ldots,n+1$ respectively.

\begin{theorem} \label{thm:numcells}
The following numbers are all equal:
\begin{itemize}
\item[(a)] the number of cells of the canonical subdivision of
$\F_G(a_1,a_2,\ldots,a_n,-\sum a_i)$,
\item[(b)] the sum 
\begin{equation} \label{eq:numpiecesLidskii}
\sum_{{\bf j}} K_{G}(j_1-\om_1,\ldots,j_n-\om_n,0),
\end{equation}
over compositions ${\bf j} = (j_1,\ldots,j_n)$ of
$m-n$ that are $\geq (\om_1,\ldots,\om_n)$ in dominance order,
\item[(c)] the number of lattice
points of the polytope $\F_{G^{\star}}(n-m,-\om_1,\ldots,-\om_n,0)$, 
\item[(d)] the volume of the polytope $\F_{G^{\star}}(1,0,\ldots,0,-1)$,
\item[(e)] the volume of the polytope $\F_{G^{\circ}}(1,0,\ldots,0,-1)$.
\end{itemize}
\end{theorem}

\begin{proof}
From the subdivision in the proof of Theorem~\ref{thmA} for $\F_G(\a)$
the number $P$ of full-dimensional cells of the subdivision is the sum given in \eqref{eq:numpiecesLidskii}. This proves the equivalence of (a) and (b).

Next we show the equality between (b) and (c). Each term in the sum in \eqref{eq:numpiecesLidskii} counts the number of integral flows on $G$ with netflow $(j_1-\om_1,\ldots,j_n-\om_n,0)$. Each such flow corresponds to an integral flow on $G^{\circ}$ with
netflow $(n-m,-\om_1,\ldots,-\om_n)$ by assigning a flow of $j_i$ to
edge $(0,i)$ for $i=1,2,\ldots,n$. Conversely, given an integral flow in $G^{\circ}$ with netflow
$(n-m,-\om_1,\ldots,-\om_n,0)$, if $j_i$ is the netflow on edge
$(0,i)$ then the integral flows on the edges of the subgraph $G$ yields an integral flow on $G$
with netflow $(j_1-\om_1,\ldots,j_n-\om_n,0)$. Thus   
\[
P = K_{G^{\star}}(n-m,-\om_1,\ldots,-\om_n,0).    
\]
This proves the equivalence of (b) and (c).

Next, the numbers in (c) and (d) are equal since   \eqref{eq:vol10-1caseOut} applied to $\F_{G^\circ}(1,0,\ldots,0,-1)$  yields
\[
\vol \F_{G^\star}(1,0,\ldots,0,-1) =K_{G^{\star}}(n-m,-\om_1,\ldots,-\om_n,0).
\]

Finally, we show the equality between the numbers in  (d) and (e) by
combining \eqref{eq:vol10-1caseOut} with the observation that 
\[
K_{G^{\star}}(n-m,-\om_1,\ldots,-\om_n,0) = K_{G^{\circ}}(n-m,-\om_1,\ldots,-\om_n,0),
\]
where $\om_i = \om_i(G^{\star}) = \om_i(G^{\circ})$ for $i=1,\ldots,n$.
\end{proof}

\begin{remark}
In Section~\ref{sec:cayleytrick} we give a second proof of the equality between (a) and (d) in 
Theorem~\ref{thm:numcells} using the {\em Cayley trick}
 \cite{HRS, Sturmfels}.  
\end{remark}

\begin{corollary}[{\cite[Thm. 1]{PS}}] \label{cor:numcellsPS}
The number of cells of the canonical subdivision of the Pitman--Stanley polytope $\F_{\Pi_n}(\a)$ is $C_n$.
\end{corollary}

\begin{proof}
By   in Theorem~\ref{thm:numcells} (a)$=$(b) the number of cells of the canonical subdivision of $\F_{\Pi_n}(\a)$ equals the sum
\[
P =  \sum_{\bf j} K_{\Pi_n}(j_1-1,\ldots,j_n-1,0).
\]
By Corollary~\ref{cor:numtypesPS} the sum on the RHS above has $C_n$ compositions {\bf j} with nonzero contribution. Each Kostant partition function in the sum has zero netflow on vertex $n+1$. Thus each such term counts integral flows on the path $1\to 2 \to \cdots \to n$. There is exactly one such integral flow, so $K_{\Pi_n}(j_1-1,\ldots,j_n-1,0)=1$ for each of the $C_n$ many compositions ${\bf j} \geq (1,\ldots,1)$.
\end{proof}

\begin{corollary} \label{thm:numcellssubdiv}
The number of cells of the canonical subdivision of $\F_{k_{n+1}}({\bf
  a})$  for $\a \in \mathbb{Z}^n_{>0}$ is $C_0C_1C_2\cdots C_{n-1}$.
\end{corollary}

\begin{proof}
For $G=k_{n+1}$ we have that $G^{\circ}=k_{n+2}$. Then by Theorem~\ref{thm:numcells} (a)$=$(d), the desired number of cells equals the volume of the CRY
polytope of size $n+1$. The result then follows by \eqref{eq:volCRY}.
\end{proof}

\begin{example}
Continuing with Example~\ref{ex:numtypecells}, the subdivision of $\F_{k_4}(1,1,1,-3)$ illustrated in Figure~\ref{fig:subdivBVpoly} has $C_1C_2=2$ cells ($t_3=2$ types of cells, each appearing once). The subdivision of $\F_{k_5}(1,1,1,1,-4)$ has $C_1C_2C_3=10$ cells of $t_4=7$ different types. 
\end{example}

\subsection{Number of words in the Lidskii volume formula} 

If we  take the Lidskii formula for the  volume of $\F_{G}({\bf
  a})$ and we look at it as a sum of {\em words} $w=w_1w_2\cdots w_n$
in the alphabet $a_1,a_2,\ldots$
(the order of letters matters), then \eqref{eq:vol} becomes
\begin{equation} \label{eq:words}
\vol(F_{G}({\bf a})) = \sum_w m(w) \cdot w_1 w_2 \cdots w_n.
\end{equation}
where $m(w)$ is the multiplicity of the word $w$. See Example \ref{??} below.  From the Lidskii
formula \eqref{eq:vol} the multiplicity is given by a Kostant
partition function
\[
m(w) = K_G(j_1-\out_1,\ldots,j_n-\out_n,0),
\]
where $j_k$ is the number of instances of the letter $a_k$ in $w$. The following
proposition gives the number of such words with multiplicity as a
volume of another flow polytope.

\begin{proposition}
For the flow polytope $\F_G(\a)$ and the words $w$ as defined above we have that 
\[
\sum_w m(w) = \vol \F_G(1,\ldots,1,-n).
\]
\end{proposition}

\begin{proof}
To count the words with multiplicity it suffices to evaluate $a_i=1$
in \eqref{eq:vol}. 
\end{proof}

For the Pitman--Stanley polytope the multiplicity of each words $w$ in
\eqref{eq:words} is  $m(w) =K_{\Pi_n}(j_1-1,\ldots,j_n-1,0)$. This value
of the Kostant partition function equals $1$ as explained in the proof
of Corollary~\ref{cor:numcellsPS}. Moreover, the words appearing in
the formula are {\em parking functions} as shown in \cite{PS}.

\begin{corollary}[{\cite[Thm. 11]{PS}}] \label{cor:charwordsPS}
For the Pitman--Stanley polytope $\F_{\Pi_n}(\a)$ we have that 
\[
\vol \F_{\Pi_n}(\a) = \sum_{(k_1,\ldots,k_n)} a_{k_1}a_{k_2}\cdots a_{k_n},
\]
where the sum is over parking functions $(k_1,\ldots,k_n)$. Thus the number of words in the Lidskii volume formula is $(n+1)^{n-1}$.
\end{corollary}

\begin{corollary}
For the flow polytope $\F_{k_{n+1}}(\a)$, the number of words with multiplicity in the Lidskii volume formula equals
\[
\sum_w m(w) = f^{(n-1,n-2,\ldots,1)}\cdot C_1C_2\cdots C_{n-1}.
\]
\end{corollary}

\begin{proof}
This number of words is exactly the volume
of the Tesler polytope $\vol\F_{k_{r+1}}({\bf 1})$ given in \eqref{eq:volTesler}.
\end{proof}

\begin{example} \label{??}
For the graph $G=k_4$, omitting from the notation the netflow on the last vertex, we have that 
\[
\vol \F_{k_4}(a_1,a_2,a_3) = \binom{3}{3,0,0} a_1^3\cdot
K_{k_4}(1,-1,0)+\binom{3}{2,1,0} a_1^2a_2 \cdot K_{k_4}(0,0,0),
\]
and the polytope subdivides into $K_{k_4}(1,-1,0)+ K_{k_4}(0,0,0)=2$ cells. In
terms of words: 
\[
\vol \F_{k_4}(a_1,a_2,a_3)  = a_1a_1a_1\cdot K_{k_4}(1,-1,0) +
(a_1a_1a_2+a_1a_2a_1+a_2a_1a_1)\cdot K_{k_4}(0,0,0), 
\]
i.e. the volume formula is given in terms of four words. 
\end{example}

\begin{remark}
It is natural to ask for a characterization of the words that appear in \eqref{eq:words}. In the
Pitman-Stanley polytope the equivalent words are parking functions
(see Corollary~\ref{cor:charwordsPS}). See \cite{AIM} for a recent characterization. 
\end{remark}

\subsection{Flow polytope with volume counted by lattice points of Pitman--Stanley polytope} \label{sec:genPS}

Given ${\bf c} := (c_1,c_2,\ldots,c_n)$ for nonnegative integers $c_i$, let $\Pi_n({\bf c})$ be the graph with vertices $[n+1]$ consisting of the path $1\to 2 \to \cdots \to n+1$ and $c_i$ multiple edges of the form $(i,n+1)$. Recall that $\Pi_n({\bf c})^{\star}$ denotes the graph $\Pi_n({\bf c})$ with an additional vertex $0$ adjacent to vertices $1,2,\ldots,n$. See Figure~\ref{fig:pincstar}. At $c_1=\ldots=c_n=1$, the graph $\Pi_n(1,\ldots,1)$ equals the graph $\Pi_n$.

\begin{figure}
    \centering
    \includegraphics{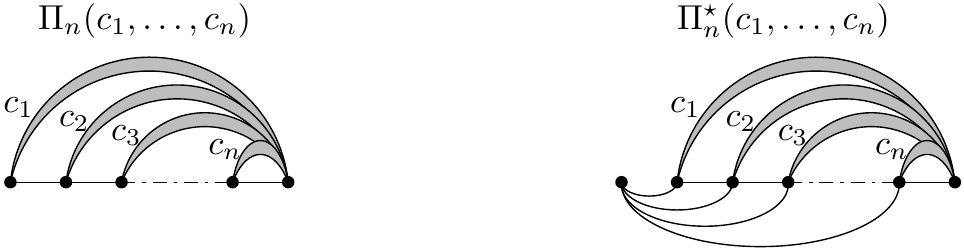}
    \caption{Example of graphs $\Pi_n({\bf c})$ and $\Pi_n({\bf
        c})^{\star}$. The gray edges with label $c_i$ indicate $c_i$
      multiple edges.}
    \label{fig:pincstar}
\end{figure}

\begin{corollary}
Let $\a = (a_1,\ldots,a_n, -\sum_i
a_i)$  and ${\bf c} = (c_1,\ldots,c_n)$ be tuples of nonnegative
integers $a_i$ and $c_i$. Then the volume and lattice points of the flow polytope
$\F_{\Pi_n({\bf c})}({\bf a})$ equal 
\begin{align}
\vol \F_{\Pi_n({\bf c})}({\bf a})  &= \sum_{\bf j} \binom{c_1+\cdots +
                                     c_n}{j_1,\ldots,j_n}
                                     a_1^{j_1}\cdots a_n^{j_n}, \label{eq:volgenPS}\\
K_{\Pi_n({\bf c})}({\bf a}) &= \sum_{\bf j} \binom{a_1+c_1}{j_1}\cdots
                              \binom{a_n+c_n}{j_n}, \label{eq:kostgenPS}\\
&= \sum_{\bf j} \multiset{a_1+1}{j_1}\multiset{a_2}{j_2}\cdots
\multiset{a_{n}}{j_{n}}, \label{eq:kostmultisetgenPS}
\end{align}
where the three sums are over weak compositions ${\bf j}=(j_1,\ldots,j_n)$ of $\sum_i c_i$ that are $\geq (c_1,\ldots,c_n)$ in dominance order.
\end{corollary}

\begin{proof}
The result follows by  Theorem~\ref{thmA} for
$G=\Pi_n({\bf c})$ with $\out_i=1$ for $i=1,\ldots,n$,
$\i_1=-1,\i_j=0$ for $j=2,\ldots,n$, and noticing that
\begin{equation} \label{eq:KgenPS}
K_{\Pi_n({\bf c})}(j_1-1,\ldots,j_n-1,0) = K_{\Pi_n}(j_1-1,\ldots,j_n-1,0) =1,
\end{equation}
where the second equality follows from the proof of Corollary~\ref{cor:numcellsPS}.
\end{proof}

\begin{corollary} \label{cor:volflowlattPS}
Let ${\bf c} = (c_1,\ldots,c_n)$ be a tuple of nonnegative integers.
and let $\Pi_n({\bf c})^{\star}$  be the graph defined above, then
\[
\vol \F_{\Pi_n({\bf c})^{\star}}(1,0,\ldots,0,-1) = \#(\PS(c_n,c_{n-1}\ldots,c_2) \,\cap\, \mathbb{Z}^{n-1}) = \det\left[
\binom{c_{i+1}+\cdots +c_{n}+1}{i-j+1}\right]_{1\leq i,j\leq n-1}.
\]
In particular, the volume is independent of $c_1$.
\end{corollary}

\begin{proof}
The result follows by combining Theorems~\ref{thm:numcells} (b)$=$(d),
\eqref{eq:KgenPS},  ~\ref{cor:corrlatptsPS} and Corollary~\ref{cor:numtypescells}.
\end{proof}

\begin{remark}
The following particular case of the previous volume formula gives a
product. When ${\bf c} = (d,\ldots,d,n)$ is a tuple of size $n+1$, by
Corollary~\ref{cor:volflowlattPS} the volume of
$\F_{\Pi_{n+1}(d,\ldots,d,c)^{\star}}(1,0,\ldots,0,-1)$ equals the
number of
lattice points of $\PS(c,d^n)$.  In \cite[Thm. 13]{PS}, this number of
lattice points has a product formula, giving
\[
\vol \F_{\Pi_{n+1}(d,\ldots,d,c)^{\star}}(1,0,\ldots,0,-1) = \frac{1}{n!} (c+1)(c+nd+2)(c+nd+3)\cdots (c+nd+n).
\]
\end{remark}

\begin{remark}
The volume of the polytope $\F_{\Pi_{n+1}({\bf c})}( {\bf 1})$ equals
  the number of {\em generalized parking functions} studied by
  Yan \cite{CY1,CY2}. Also, the polytope $\F_{\Pi_{n+1}({\bf
      1})^{\star}}(1,1,\ldots,1,-n-1)$ appears in \cite{AIM} and is
  called the {\em Caracol polytope}. Its volume equals $C_{n-1}\cdot (n+1)^{n-1}$. 
\end{remark}

We finish our treatment of the flow polytope $\F_{\Pi_{n+1}({\bf c})}(\a)$
by proving that its Ehrhart polynomial has positive 
coefficients. This was known for the Pitman--Stanley polytope
\cite[Eq. (33)]{PS}. For more on positivity of coefficients of Ehrhrat
polynomials see e.g. \cite{CL,Lsurvey}.

\begin{corollary} \label{cor:ehrhartposgenPS}
The Ehrhart polynomial of $\F_{\Pi_{n+1}({\bf c})}(\a)$ has positive
coefficients.
\end{corollary}

\begin{proof}
The result follows by the formula \eqref{eq:kostmultisetgenPS} for the
Ehrhart polynomial $K_{\Pi_{n+1}({\bf c})}(t\cdot \a)$.
\end{proof}

\begin{remark}
The positivity in $t$ of the polynomial $K_{G}(t\cdot \a)$ is not apparent from
either \eqref{eq:kost} or \eqref{eq:kostmultiset}. There are examples
of graphs $G$ where $K_G(t,0,\ldots,0,-t)$ has negative coefficients
in $t$ \cite[Sec. 4.4]{Lsurvey}. However, positivity
in equation \eqref{eq:kostmultiset} holds if $a_j\geq \i_j$ for
$j=1,\ldots,n$ or if the Kostant partition function on the LHS
of these equations is replaced by $1$ and graph $G$ verifies
 (by
Corollary~\ref{cor:ehrhartposgenPS} with $c_i = \om_i\geq 0$).
\end{remark}

\section{The Cayley trick for flow polytopes} \label{sec:cayleytrick}

Corollary~\ref{cor:vol10-1case} and the Lidskii volume formula \eqref{eq:vol} express the volume of flow polytopes in terms of the number of lattice points of several related flow polytopes.  The volumes of root polytopes and integer points of  generalized permutahedra obey a similar relation, as shown in \cite[\S 14]{P} by Postnikov.  Postnikov used the {\em Cayley trick}  \cite{HRS, Sturmfels} to give the volume of root polytopes  in terms of the number of lattice points of generalized permutahedra.  The first author and St. Dizier  proved a relation between volumes of flow polytopes and integer points of generalized permutahedra \cite{MStD}.
In this section we  use the Cayley trick  to give a second proof of Theorem~\ref{thm:numcells}. It would be interesting to use this technique to fully rederive the Lidskii formulas.  

We follow the notation in \cite[\S 14]{P}. Given a polytope $P$, its polytopal subdivisions form a poset by refinement whose minimal elements correspond to triangulations. Given a $d$-dimensional Minkowski sum $Q:= P_1 + \cdots + P_n$, a {\em Minkowski cell} of $Q$  is a polytope $B_1 + \cdots + B_n$ where $B_i$ is a convex hull of a subset of vertices of $P_i$. A {\em mixed subdivision} of $Q$ is a decomposition of $Q$ into Minkowski cells, such that the intersection of two such cells is a common face. These subdivisions form a poset by refinement whose minimal elements are called {\em fine mixed subdivisions}.

Let $P_1,\ldots,P_n$ be polytopes in $\mathbb{R}^m$, and by abuse of notation we say that $\mathbb{R}^{n+m}$ has a standard basis $\sfe_1,\ldots, \sfe_n, \sfe'_1,\ldots,\sfe'_m$. The {\em Cayley embedding} of polytopes $P_1,\ldots,P_n$ in $\mathbb{R}^m$ is the polytope $\mathcal{C}(P_1,\ldots,P_n)$ given by the convex hull of $\sfe_i \times P_i$ for $i=1,\ldots,n$.

\begin{proposition}[The Cayley trick  \cite{HRS}] \label{trick}
 For any positive parameters $a_1,\ldots,a_n$ with $\sum a_i=1$, any polytopal subdivision of $\mathcal{C}(P_1,\ldots,P_n)$ intersected by $(a_1,\ldots,a_n) \times \mathbb{R}^m$ gives a mixed subdivision of $a_1P_1+\cdots + a_nP_n$. This correspondence gives a  poset isomorphism between the poset of polytopal subdivisions of $\mathcal{C}(P_1,\ldots,P_n)$ and the poset of mixed subdivision of $a_1P_1+\cdots + a_nP_n$, both ordered by refinement. \end{proposition}

Recall that by Proposition~\ref{prop:flow-minkowski} the flow polytope $\F_G(\a)$, $\a \in \Z^n_{\geq 0}$, is the Minkowski sum \eqref{eq:flow-minkowski} of flow polytopes $\F_G(e_i-e_{n+1})$ for $i=1,\ldots,n$.  Also recall that for a   graph $G$ on the vertex set $[n+1]$, we let $G^{\star}$ be the graph obtained from $G$ by adding a vertex $0$ adjacent to vertices $i=1,2,\ldots,n$.

\begin{proposition} \label{cay}
 The Cayley embedding $\mathcal{C}\left(\F_G(e_1-e_{n+1}),\F_G(e_2-e_{n+1}),\ldots,\F_G(e_n-e_{n+1})\right)$ is the flow polytope $\F_{G^{\star}}(e_1-e_{n+2})$.
\end{proposition}

\begin{proof}
 $\mathcal{C}\left(\F_G(e_1-e_{n+1}),\F_G(e_2-e_{n+1}),\ldots,\F_G(e_n-e_{n+1})\right)$ is the convex hull of ${\sf e}_i \times \F_G(e_i-e_{n+1})$ for $i=1,2,\ldots,n$. Regard ${\sf e_i}$ as a unit flow on the edge $(0,i)$. Since by Proposition~\ref{prop:vertFG1000} the vertices of $\F_G(e_i-e_{n+1})$ are unit flows supported on the directed paths from vertex $i$ to vertex $n+1$,  by concatenating these paths to the edge $(0,i)$ we obtain directed paths in $G^{\star}$ of the form $0\to i \to \cdots \to n+1$. Doing these concatenations for $i=1,\ldots,n$ yields all directed paths from vertex $0$ to vertex $n+1$ in $G^{\star}$. By Proposition~\ref{prop:vertFG1000} the unit flows on such paths give the vertices of the flow polytope $\F_{G^{\star}}(1,0,\ldots,-1)$.
\end{proof}

\begin{corollary} \label{cor:CayleyTrickFP}
For $a_1,\ldots,a_n>0$, mixed subdivisions of
$\F_G(a_1,\ldots,a_n,-\sum_i a_i)$ are in bijection with polytopal subdivisions of $\F_{G^{\star}}(e_1-e_{n+2})$. In particular fine mixed subdivisions of the former are in bijection with triangulations of the latter. 
\end{corollary}

\begin{proof} By \eqref{eq:flow-minkowski} we have that  $\F_G(\a)= a_1\F_G(e_1-e_{n+1})+\cdots+a_n \F_G(e_n-e_{n+1})$. By applying Propositions \ref{trick} and \ref{cay} we obtain the desired bijection  by intersecting polytopal subdivisions of $\F_{G^{\star}}(e_1-e_{n+2})$ with the subspace $(a_1/s,\ldots,a_n/s) \times \mathbb{R}^m$ where $s = \sum_i a_i$ followed by stretching the intersection by a factor of $s$.
\end{proof}

Next, we relate this application of the Cayley trick to $\F_G(\a)$
with the canonical subdivision of this flow polytope. The next result shows that this subdivision is a
fine mixed subdivision of $\F_G(\a)$.

\begin{lemma} \label{lemm:canonicalisfinemixed}
For the polytope $\F_G(\a)$ the canonical subdivision is a fine mixed subdivision.
\end{lemma}

\begin{proof} 
First we show that the canonical subdivision is a mixed
subdivision. By expressing $\F_G(\a)$ as
the Minkowski sum \eqref{eq:flow-minkowski} we see that each compounded
reduction (CR) on vertex $i$ subdivides the polytopes $a_j\cdot
\F_G(e_j-e_{n+1})$,  $j \in [n]$ (some of them trivially). Thus,  the subdivision of $\F_G(\a)$ by  a
CR is a mixed subdivision. Since the canonical subdivision is obtained by executing compounded reductions in a specified order,  we obtain that the 
canonical subdivision is a mixed subdivision. 

To see that the canonical subdivision is fine, we note that the pieces of
the canonical subdivision are the polytopes $\F_{G(\m)}(\a)$ for the
graphs $G(\m)$ defined in Section~\ref{sec:Gms}. Since these graphs
only have edges of the form $(i,n+1)$ then \eqref{eq:flow-minkowski}
applied to $\F_{G(\m)}(\a)$ expresses this polytope as a Minkowski sum
of simplices
\[
\F_{G(\m)}(\a)  = a_1\Delta_{m_1-1} + a_2 \Delta_{m_1-1} + \cdots +
a_n \Delta_{m_n-1},
\]
where $a_i \Delta_{m_i+1} \subseteq a_i \F_G(e_i-e_{n+1})$ as
explained in Remark~\ref{rem:interpretationFGT}. Moreover, the sum of the dimensions of the unimodular simplices in the above equation is the dimension of $\F_{G(\m)}(\a)$.  Thus we see  that the canonical subdivision is a minimal element in the
poset of mixed subdivisions of $\F_G(\a)$.
\end{proof}

We are now ready to give a second proof of Theorem~\ref{thm:numcells}
(a) $\leftrightarrow$ (d) without using the Lidskii formula
\eqref{eq:vol}. 

\begin{proof}[Second proof of Thm.~\ref{thm:numcells} (a) $\leftrightarrow$ (d)]
By Lemma~\ref{lemm:canonicalisfinemixed} the number $P$ of cells in
the canonical subdivision equals the number of cells in a fine mixed
subdivision of $\F_G(\a)$. By Corollary~\ref{cor:CayleyTrickFP} the number of cells in a fine
mixed subdivision of $\F_G(\a)$ is the number of simplices in a triangulation of $\F_{G^{\star}}(e_1-e_{n+2})$, i.e. the normalized volume of this flow polytope. 
\end{proof}

\begin{figure}
\includegraphics{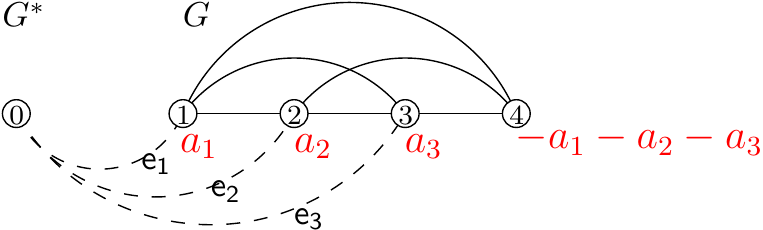}
\caption{The Cayley embedding of the flow polytope of a graph $G$ is equivalent to the flow polytope of the graph $G^{\star}$ obtained from $G$ by adding a vertex $0$ with edges $(0,i)$ for $i=1,\ldots,n$.}
\label{fig:exCayleyEmbedding}
\end{figure}

\subsection*{Acknowledgements}
We thank Alex Postnikov and Richard Stanley for sharing their insights into 
triangulating $\F_G(1,0,\ldots,0,-1)$, which served as the starting
point of our work. We also thank Federico Castillo, Sylvie Corteel, Rafael Gonz\'alez D'L\'eon,
Fu Liu and Michele Vergne for helpful comments and suggestions.

\appendix
\section{Examples of canonical subdivisions and Lidskii formulas} \label{sec:appendix}

\begin{example}
For the graph $G$ with vertices $[3]$ and edges
$\{(1,2),(1,2),(2,3),(2,3)\}$ (see Figure~\ref{fig:exflowpolys},
left) we have that $(\om_1,\om_2) = (1,1)$.  The basic reduction tree for $\F_G(1,1,-2)$ is given in
Figure~\ref{fig:brt_ex1}, left. The Lidskii volume formula~\eqref{eq:vol} gives
\[
\vol \F_G({\bf 1}) = \binom{2}{1,1} K_G(1-1,1-1,0) + \binom{2}{2,0}
K_G(2-1,0-1,0) = 2 \cdot 1 + 1 \cdot 2 =4.
\]
The first Lidskii lattice point formula~\eqref{eq:kost} gives
\[
K_G(1,1,-2) = \binom{2}{1}\binom{2}{1} K_G(0,0,0) +
\binom{2}{2}\binom{2}{0} K_G(1,-1,0) = 4\cdot 1 + 1\cdot 2 = 6.
\]
Since $(\i_1,\i_2) = (-1,1)$, the second Lidskii lattice point formula~\ref{eq:kostmultiset} gives
\[
K_G(1,1,-2) = \multiset{2}{1} \multiset{0}{1} K_G(0,0,0) +
\multiset{2}{2}\multiset{0}{0} K_G(1,-1,0) = 0 + 3\cdot 2 = 6. 
\]
The subdivision yields $K_G(1,-1,0)=2$ cells of one type with three lattice points
and $K_G(0,0,0)=1$ cell of another type with four lattice
points. Depending on how the lattice points of the common facets are
counted, we obtain the two formulas above. See
Figure~\ref{fig:brt_ex1}, right.
\end{example}

\begin{figure}
\includegraphics[scale=0.7]{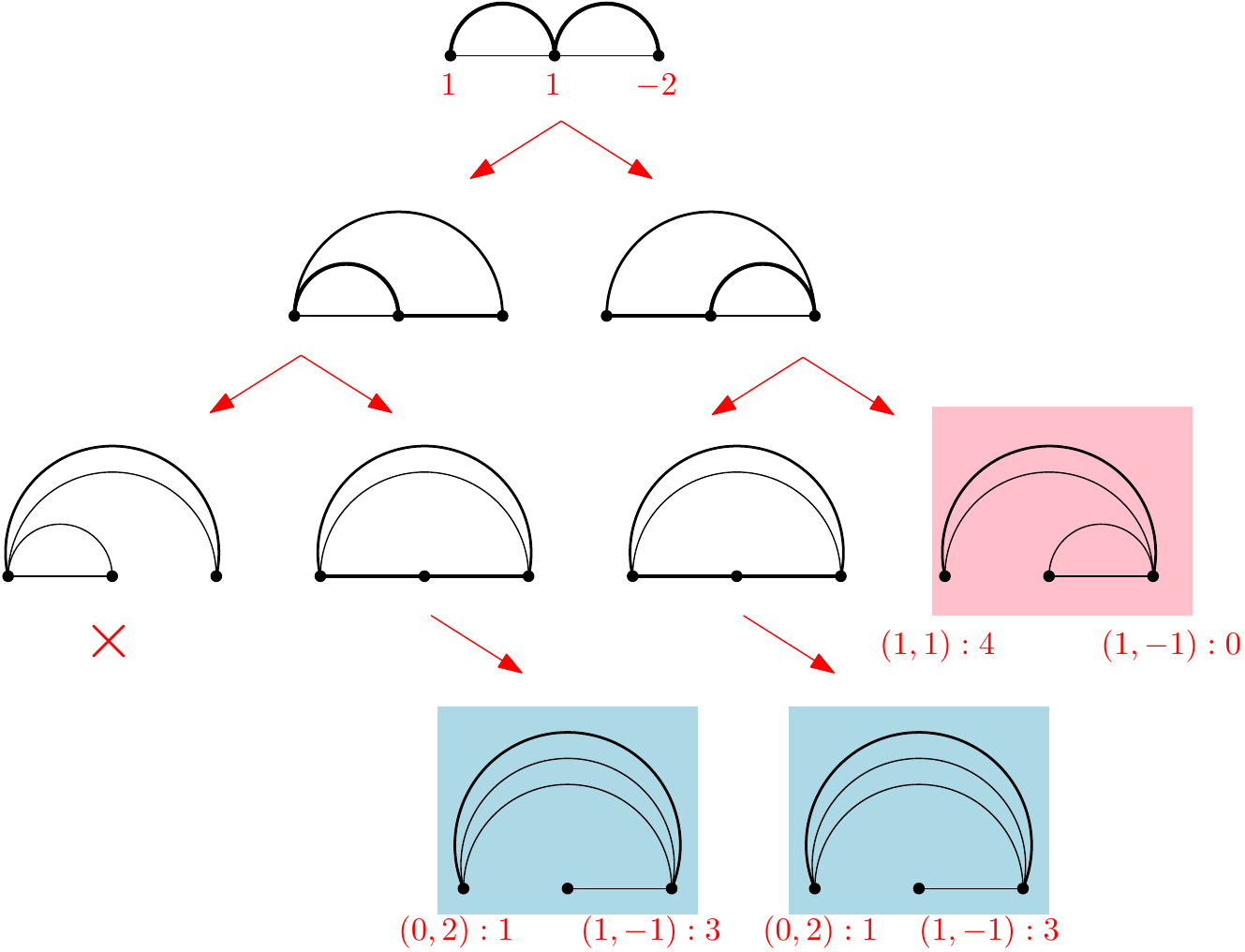}
\quad 
\includegraphics[scale=0.6]{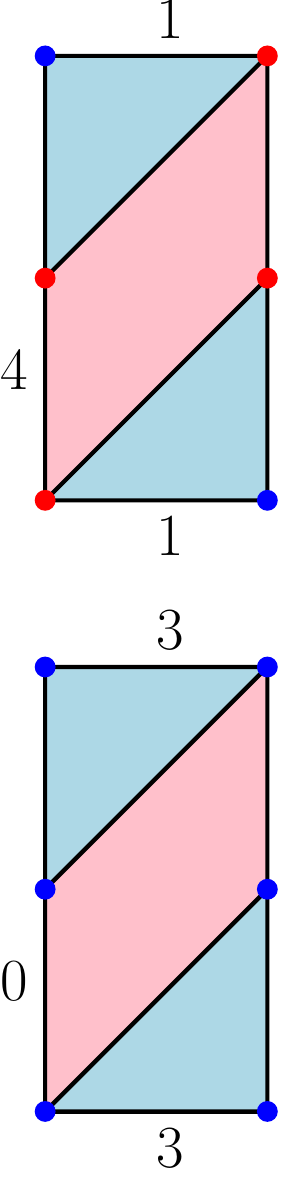}
\caption{Left: the basic reduction tree for the polytope
  $\F_{G}(1,1,-2)$ where $G$ has edges
  $\{(1,2),(1,2),(2,3),(2,3)\}$. Below each leaf $G(\m)$, the lattice
  point contribution of the corresponding cells is given in the form ${\bf v}: K_{G(\m)}({\bf v})$ where ${\bf v} ={\bf a} - {\bf G(\m)}$
  and ${\bf v} = {\bf a} - {\bf G(\m)}'$.  Right: the canonical subdivision of the
  polytope given by the reduction tree  illustrating \eqref{eq:kost}
  and \eqref{eq:kostmultiset} . The lattice points are colored
  according to the contribution of each cell.}
\label{fig:brt_ex1}
\end{figure}

\begin{figure}
\includegraphics[scale=0.6]{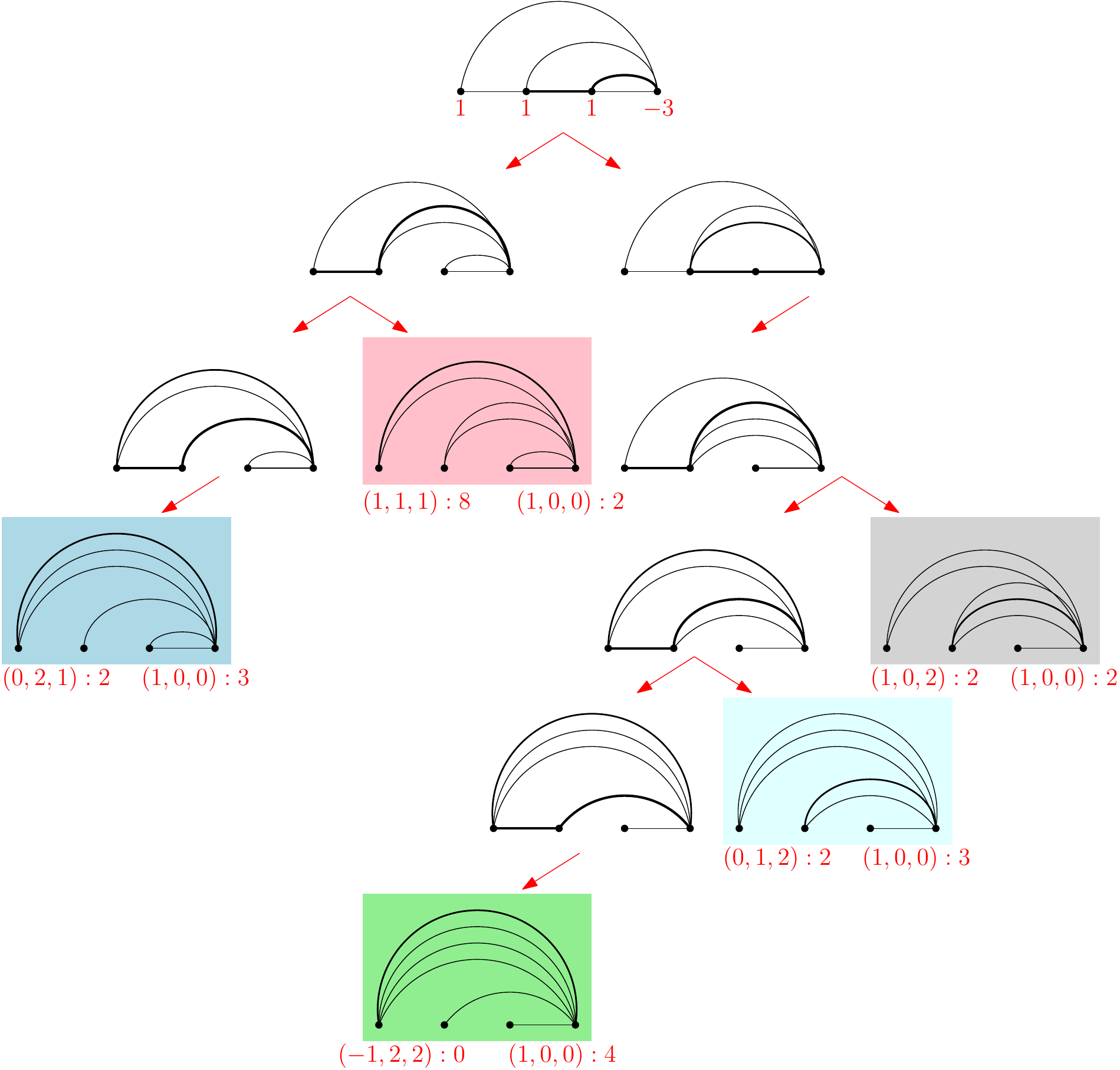} 
\quad
 \includegraphics[scale=0.5]{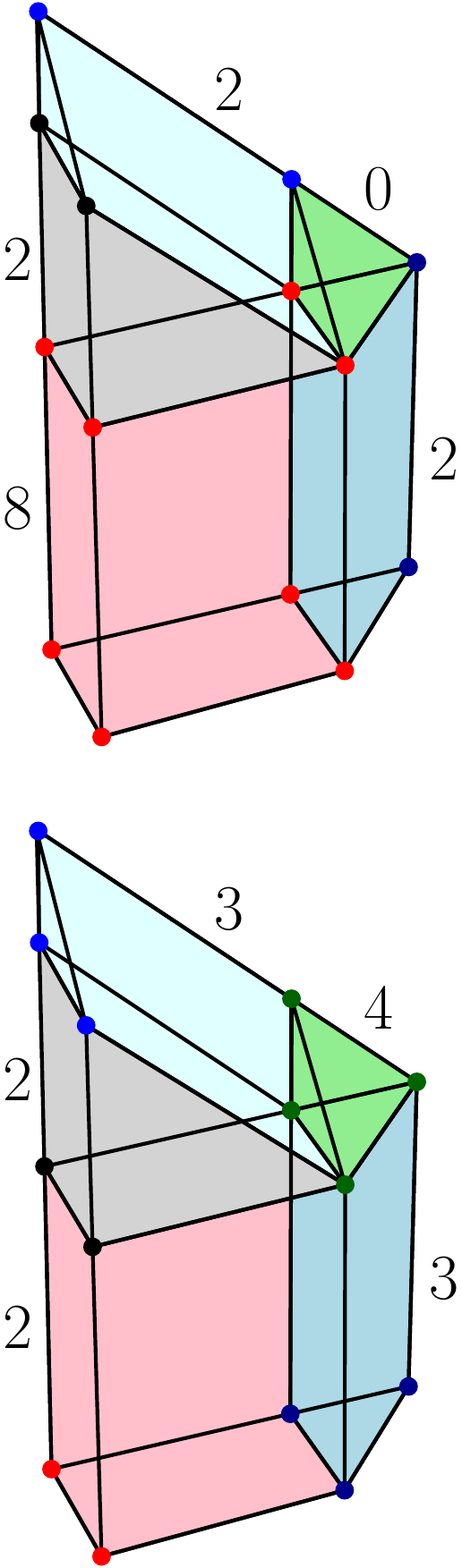}
\caption{Left: the basic reduction tree for the polytope
  $\F_{\PS_3}(1,1,1,-3)$. Below each leaf $G(\m)$, the lattice
  point contribution of the corresponding cells is given in the form ${\bf v}: K_{G(\m)}({\bf v})$ where ${\bf v} ={\bf a} - {\bf G(\m)}$
  and ${\bf v} = {\bf a} - {\bf G(\m)}'$.  Right: the canonical subdivision of the
  polytope given by the reduction tree  illustrating \eqref{eq:kost}
  and \eqref{eq:kostmultiset} . The lattice points are colored
  according to the contribution of each cell.}
\label{fig:brt_PS}
\end{figure}

\begin{example}
For the graph $\PS_3$ (see Figure~\ref{fig:exflowpolys},
center) we have that $(\om_1,\om_2,\om_3) = (1,1,1)$.  The basic reduction tree for $\F_{\PS_3}(1,1,1,-3)$ is given in
Figure~\ref{fig:brt_PS}, left. Since $K_{\PS_3}(a,b,c,0)=1$ or $0$, then the Lidskii volume formula~\eqref{eq:vol} gives
\[
\vol \F_{\PS_3}({\bf 1}) = \binom{3}{2,0,1} + \binom{3}{1,1,1}
+ \binom{3}{3,0,0}  +
\binom{3}{2,1,0}  + \binom{3}{1,2,0}  = 3+6+1+3 +3 = 16.
\]
The first Lidskii lattice point formula~\eqref{eq:kost} gives
\[
K_{\PS_3}(1,1,1,-3) = \binom{2}{0}\binom{2}{1}\binom{2}{2} +
\binom{2}{1}^3 + \binom{2}{0}\binom{2}{1}\binom{2}{2} +
\binom{2}{0}\binom{2}{1}\binom{2}{2} = 2 + 8 + 2+2=14.
\]
Since $(\i_1,\i_2,\i_3) = (-1,0,0)$, the second Lidskii lattice point formula~\ref{eq:kostmultiset} gives
\begin{multline}
K_{\PS_3}(1,1,1,-3) =  \multiset{2}{2}\multiset{1}{0}\multiset{1}{1}+\multiset{2}{1}\multiset{1}{1}\multiset{1}{1} +
\multiset{2}{3}\multiset{1}{0}\multiset{1}{0} +\\
+ 
\multiset{2}{2}\multiset{1}{1}\multiset{1}{0}  +
\multiset{2}{1}\multiset{1}{2}\multiset{1}{0}  = 3 +2+ 4+3 + 2 = 14. 
\end{multline}
The subdivision yields five cells of different types. Depending on how the lattice points of the common facets are
counted, we obtain the two formulas above. See
Figure~\ref{fig:brt_PS}, right.
\end{example}

\bibliography{biblio-kir}
\bibliographystyle{plain}

\end{document}